\documentclass[11pt,twoside]{article}

\usepackage{mathrsfs,amsfonts,amsmath,amssymb}
\usepackage{latexsym}
\newtheorem{satz}{Theorem}
\newtheorem{proposition}[satz]{Proposition}
\newtheorem{theorem}[satz]{Theorem}
\newtheorem{lemma}[satz]{Lemma}
\newtheorem{definition}[satz]{Definition}
\newtheorem{corollary}[satz]{Corollary}
\newtheorem{remark}[satz]{Remark}

\newtheorem{exc}[satz]{Exercise}

\def\no{\noindent}

\def\eps{\varepsilon}
\def\_phi{\varphi}

\def\a{\alpha}

\def\v{\vec}
\def\F{{\mathbb F}}

\def\m{\times}
\def\t{\tilde}

\def\ov{\overline}

\def\C{{\mathbb C}}
\def\R{{\mathbb R}}
\def\E{\mathsf {E}}
\def\T{{\mathbb T}}

\def\Z_N{{\mathbb Z}_N}
\def\Z{{\mathbb Z}}
\def\N{{\mathbb N}}

\def\Gr{{\mathbf G}}

\def\D{{\mathbb D}}

\def\Spec{{\rm Spec\,}}

\def\oM{{\rm M}}

\def\oT{{\rm T}}

\def\supp{{\rm supp\,}}
\def\tr{{\rm tr\,}}

\def\G{\Gamma}
\def\FF{\widehat}
\def\c{\circ}
\def\D{\Delta}
\def\Cf{{\mathcal C}}

\def\gs{\geqslant}
\def\T{\mathsf {T}}

\author{Shkredov I.D.}
\title{ An introduction to higher
energies
and sumsets
}
\date{}
\begin{document}
\maketitle

\begin{center}
 Annotation.
\end{center}

{\it \small
     These notes basically contain a material of two mini--courses  which were read in G\"{o}teborg in April 2015 during the author visit of Chalmers \& G\"{o}teborg universities and in Beijing in November 2015 during "Chinese--Russian Workshop on Exponential Sums and Sumsets".
     The
     %notes
     article
     is a short introduction to a new area of Additive Combinatorics which is connected which so--called the higher sumsets as well as with the higher energies.
     We hope the notes will be helpful for a reader who is interested in the field.
}
\\
%\\
%\\

\section{Introduction}
\label{sec:introduction}

Let $\Gr = (\Gr, +)$ be a group with the group operation $+$.
By letters $A,B,C, \dots$ we will denote arbitrary subsets of the group $\Gr$.
Define the sumset and, similarly, the difference  set of two sets $A,B$ as
$$
    A+B := \{ a+b ~:~ a\in A,\, b\in B\} \quad A-B := \{ a-b ~:~ a\in A,\, b\in B\} \,.
$$
Of course, one can iterate the sumsets/difference sets, obtaining sums $A_1+A_2+\dots+A_k$ and so on.
If $A_1=\dots A_k = A$ then we write $kA$ for the sumset of $k$ sets $A$.
The subject of Additive Combinatorics is any combinatorics which can be expressed with the help of the group operation $+$.
Typical questions in the field are finding different connections between the sizes of sumsets, different sets, cardinalities of  its iterations and so on.

In the notes we investigate  higher sumsets and generalized convolutions, which are closely connected with the new object.
Also we study higher moments of these convolutions (higher energies),
which generalize a classical notion of the additive energy.
Such quantities appear in many problems of Additive Combinatorics as well as in Number Theory. In our investigation
we use different approaches including basic combinatorics, Fourier
analysis and the eigenvalues method to establish basic properties
of the higher energies. Also we provide a sequence of applications of
the higher energies to Additive Combinatorics and Number Theory.

The first part of the notes (section 3)  is  based on paper \cite{SS1} and partially on \cite{SV}, the second one uses \cite{SS3} and partially \cite{SS2}, \cite{RR-NS}.
%These notes are based on paper \cite{SS1}, which correspond as well as previous and latter
The last part has roots in paper \cite{s}, where nevertheless a "dual"\, Fourier notation was used.
After that the point of view was developed in articles \cite{SS1}, \cite{s_ineq}, \cite{s_mixed} and others.
As for applications sections \ref{sec:lecture2}, \ref{sec:lecture3} contain modern estimate for the size of the difference sets of convex sets and a new upper bound for Heilbronn's exponential sum. Also we discuss a structural $\E_2, \E_3$ result and the best upper bound for the additive energy of a multiplicative subgroup in the last section.
Finally, the notes contain some instructive exercises.

%The author is grateful to Sergey Konyagin and Misha Rudnev for useful discussions
%and, especially, Tomasz Schoen for very useful and fruitful explanations and  discussions.
%Also I
%acknowledge
%%is grateful to
%%thanks
%Institute IITP RAS for providing me with  excellent working conditions.

\section{Notation}
\label{sec:definitions}

%In the paper we use the same letter to denote a set $S\subseteq \f_p$
%and its characteristic function $S:\f_p \rightarrow \{0,1\}.$
%By $|S|$ denote the cardinality of $S$.

%We conclude with few comments regarding the notation used in this paper.
By $\Gr = (\Gr, +)$ we denote a group with the group operation $+$.
For a positive integer $n,$ we set $[n]=\{1,\ldots,n\}.$
All logarithms are base $2.$ Signs $\ll$ and $\gg$ are the usual Vinogradov's symbols.
%Finally,
With a slight abuse of notation we use the same letter to denote a set $S\subseteq \Gr$
and its characteristic function $S:\Gr\rightarrow \{0,1\}$,
in other words $S(x) = 1$, $x\in S$ and $S(x)=0$ otherwise.

Let $f,g : \Gr \to \C$ be two functions.
Put
\begin{equation}\label{f:convolutions}
    (f*g) (x) := \sum_{y\in \Gr} f(y) g(x-y) \quad \mbox{ and } \quad
        %(f\circ g) (x) := \sum_{y\in \Gr} \ov{f(y)} g(y+x)
        (f\circ g) (x) := \sum_{y\in \Gr} f(y) g(y+x)
\end{equation}
 Clearly,  $(f*g) (x) = (g*f) (x)$ and $(f\c g)(x) = (g \c f) (-x)$, $x\in \Gr$.
 The $k$--fold convolution, $k\in \N$  we denote by $*_k$,
 so $*_k := *(*_{k-1})$.
%Write
Put
$\E^{+}(A,B)$ for the {\it additive energy} of two sets $A,B \subseteq \Gr$
(see e.g. \cite{tv}), that is
$$
    \E^{+} (A,B) = |\{ a_1+b_1 = a_2+b_2 ~:~ a_1,a_2 \in A,\, b_1,b_2 \in B \}| \,.
$$
If $A=B$ then we simply write $\E^{+} (A)$ instead of $\E^{+} (A,A).$
Clearly,
\begin{equation*}\label{f:energy_convolution}
    \E^{+} (A,B) = \sum_x (A*B) (x)^2 = \sum_x (A \circ B) (x)^2 = \sum_x (A \circ A) (x) (B \circ B) (x)
    \,.
\end{equation*}
Sumsets and energies are connected by the Cauchy--Schwarz inequality
\begin{equation}\label{f:E_sumsets_CS}
    |A|^2 |B|^2 \le \E^{+} (A,B) |A\pm B| \,.
\end{equation}
Note also that
\begin{equation}\label{f:E_CS}
    \E^{+} (A,B) \le \min \{ |A|^2 |B|, |B|^2 |A|, |A|^{3/2} |B|^{3/2} \} \,.
\end{equation}
In the same way define the {\it multiplicative energy} of two sets $A,B \subseteq \Gr$
$$
    \E^{\times} (A,B) = |\{ a_1 b_1 = a_2 b_2 ~:~ a_1,a_2 \in A,\, b_1,b_2 \in B \}| \,.
$$
Certainly, multiplicative energy $\E^{\times} (A,B)$ can be expressed in terms of multiplicative convolutions,
similar to (\ref{f:convolutions}).
Usually we will use the additive energy and write $\E(A,B)$ instead of $\E^{+} (A,B)$.
Sometimes we
% write
put
$\E(f,g) = \sum_x (f\c f) (x) (g\c g) (x)$ for two arbitrary functions $f,g : \Gr \to \C$.

\bigskip

In the lecture notes we will use Fourier analysis, although it is not main topic of our course.
%So,
Nevertheless,
let us recall the required definitions, for more details see \cite{Rudin_book}.

Let $\Gr$ be an abelian group.
If $\Gr$ is finite then denote by $N$ the cardinality of $\Gr$.
It is well--known~\cite{Rudin_book} that the dual group $\FF{\Gr}$ is isomorphic to $\Gr$ in the case.
Let $f$ be a function from $\Gr$ to $\mathbb{C}.$  We denote the Fourier transform of $f$ by~$\FF{f},$
\begin{equation}\label{F:Fourier}
  \FF{f}(\xi) =  \sum_{x \in \Gr} f(x) e( -\xi \cdot x) \,,
\end{equation}
where $e(x) = e^{2\pi i x}$
and $\xi$ is a homomorphism from $\FF{\Gr}$ to $\R/\Z$ acting as $\xi : x \to \xi \cdot x$.
We rely on the following basic identities
\begin{equation}\label{F_Par}
    \sum_{x\in \Gr} |f(x)|^2
        =
            \frac{1}{N} \sum_{\xi \in \FF{\Gr}} \big|\widehat{f} (\xi)\big|^2 \,,
\end{equation}
\begin{equation}\label{F_Par'}
            \sum_{x\in \Gr} f(x) \overline{g(x)}
                = \frac{1}{N} \sum_{\xi \in \FF{\Gr}} \FF{f}(\xi) \overline{\FF{g}(\xi)} \,,
\end{equation}
\begin{equation}\label{svertka}
    \sum_{y\in \Gr} \Big|\sum_{x\in \Gr} f(x) g(y-x) \Big|^2
        = \frac{1}{N} \sum_{\xi \in \FF{\Gr}} \big|\widehat{f} (\xi)\big|^2 \big|\widehat{g} (\xi)\big|^2 \,,
\end{equation}
and
\begin{equation}\label{f:inverse}
    f(x) = \frac{1}{N} \sum_{\xi \in \FF{\Gr}} \FF{f}(\xi) e(\xi \cdot x) \,.
\end{equation}
Further, we have
\begin{equation}\label{f:F_svertka}
    \FF{f*g} = \FF{f} \FF{g} \quad \mbox{ and } \quad \FF{f \circ g} = \FF{f^c} \FF{g} = \ov{\FF{\ov{f}}} \FF{g} \,,
    %(\F{fg}) (x) = \frac{1}{N} (\F{f} * \F{g}) (x) \,.
\end{equation}
where for a function $f:\Gr \to \mathbb{C}$ we put $f^c (x):= f(-x)$.
% It is unimportant but write for definiteness
% $$(f \c_k f) (x) := \sum_{y_1,\dots,y_k} f(y_1) \dots f(y_k) f(x+y_1+\dots+y_k)\,.$$
Clearly, $S$ is the characteristic function of  a set iff
\begin{equation}\label{f:char_char}
    \FF{S} (x) = N^{-1} (\ov{\FF{S}} \c \FF{S}) (x) \,.
\end{equation}

In terms of Fourier transform the common additive energy of two sets $A, B \subseteq \Gr$, $|\Gr| < \infty$ can be expressed as follows
$$
    \E(A,B) = \frac{1}{|\Gr|} \sum_{\xi \in \FF{\Gr}} |\FF{A} (\xi)|^2 |\FF{B} (\xi)|^2 \,,
$$
see formula (\ref{svertka}).

\section{Higher sumsets}
\label{sec:lecture1}

%\bigskip
%$\hfill\Box$

An usual sumset $A+B$ can be considered as the  set of nonempty intersections
$$
    A+B = \{ s \in \Gr ~:~ A \cap(s-B) \neq \emptyset \} \,.
$$
It is more convenient for us to have deal with the symmetric case $A=B$ and moreover we consider the difference sets instead of the sumsets (basically because the difference sets have more structure than the sumsets).
Thus
\begin{equation}\label{f:A-A_via_A_s}
    A-A = \{ s \in \Gr ~:~ A \cap(A-s) \neq \emptyset \} \,.
\end{equation}
If we put $A_s := A \cap (A-s)$ then the set $A-A$ is exactly the set of all $s$ such that $A_s \neq \emptyset$.
In view of formula (\ref{f:A-A_via_A_s}) a natural multidimensional generalization of the difference sets is
\begin{equation}\label{f:A-A_via_A_s'}
    \{ \v{s} = (s_1,\dots,s_k) \in \Gr^k ~:~ A \cap(A-s_1) \cap \dots \cap (A-s_k) \neq \emptyset \} \,.
\end{equation}
Again, putting $A_{\v{s}} := A \cap(A-s_1) \cap \dots \cap (A-s_k)$, we see that the set of $\v{s}$ which are defined in (\ref{f:A-A_via_A_s}) coincide with the set of vectors $\v{s}$ such that $A_{\v{s}} \neq \emptyset$.
One can easily check that our multidimensional difference set (or equivalently, the {\it higher difference set}) is just
$$
    A^k - \D_k (A) = \{ (s_1,\dots,s_k) \in \Gr^k ~:~ A \cap(A-s_1) \cap \dots \cap (A-s_k) \neq \emptyset \} \,,
$$
where
$$
    \D (A) = \D_k (A) = \{ (a,\dots,a) \in \Gr^k ~:~ a\in A \}
$$
is a {\it diagonal set}.
We also write $\D(a) = \D_k (a) = (a,\dots,a) \in \Gr^k$ for $a\in A$.
Thus, any higher difference set is just an usual difference of two but rather specific sets, namely, the Cartesian product $A^k$ and very thin diagonal set $\D_k (A)$.

\bigskip

It is well--known that the sumsets/difference sets can be considered as projections of $A\times A$ along the lines $y=c-x$ and $y=c+x$, correspondingly, onto any of two axes.

\begin{exc}
    Using projections show that for the usual Cantor one--third  set $\mathbf{K}_3$ one has $\mathbf{K}_3 + \mathbf{K}_3 = [0,2]$.
\end{exc}

Similarly, higher difference sets are projections of $A^k$ along the lines
\begin{displaymath}
\left\{ \begin{array}{ll}
x_1 = t + c_1\\
\dots \dots \dots \dots \\
x_k = t + c_k
\end{array} \right.
\end{displaymath}
So, the line above intersects $A^k$ iff $(c_1,\dots,c_k) \in A^k - \D (A)$.
Projections along hyperspaces and connected
%quantities
energies
$\T_k (A)$ which are very popular in Analytical Number Theory  will be considered in the next section.
What can we say about another projections of intermediate dimensions?
The question is widely open.

\bigskip

We
%conclude
continue
the section by Ruzsa's triangle inequality, see e.g. \cite{tv}, which is an important tool of Additive Combinatorics.
Interestingly, that our proof (developing some ideas of paper \cite{SS1})
%, \cite{B_RN_S})
%does not require any mapping as usual.
describes the situation when the triangle inequality is sharp.
Namely, the rough  equality can be only
%happen
if $|B \cap (A-z) - C| \approx |C|$ for many $z\in A-B$.

\begin{lemma}
    Let $A,B,C\subseteq \Gr$ be any sets.
    Then
\begin{equation}\label{f:triangle_my}
    |C| |A-B| \le |A\times B - \Delta (C)| \le |A-C| |B-C| \,.
\end{equation}
    In particular
\begin{equation}\label{f:triangle_my-}
    |C| |A-B| \le |A-C| |B-C| \,.
\end{equation}
\label{l:triangle_my}
\end{lemma}
\begin{proof}
We give two proofs of (\ref{f:triangle_my-}).
The first proof is standard and the second one will follow from (\ref{f:triangle_my}).

{\it The first proof.}
By the definition of the difference set $A-B$ for any $x\in A-B$ there are $a_x \in A$ and $b_x \in B$ such that $x=a_x-b_x$.
Of course it can be exist several pairs $(a_x,b_x)$ with the property but we fix just one of them somehow and do not consider another pairs.
Now let us define the map
$$
    \_phi : C \times (A-B) \to A\times B -\D(C) \subseteq (A-C) \times (B-C)
$$
by the rule
$\_phi (c,x) = (a_x - c, b_x -c)$, $c\in C$, $x\in A-B$.
It is easy to see that the map is injective.
Indeed if we have
\begin{equation}\label{tmp:14.11.2015_1}
    \_phi (c,x) = (a_x - c, b_x -c) = (a_x' - c', b_x' -c') = \_phi (c',x') \,, \quad c,c'\in C,\, x,x'\in A-B
\end{equation}
then subtracting the second coordinate from the first one, we obtain $x=a_x-b_x = a_x'-b_x'=x'$.
By our definition of $a_x,b_x$ for any $x\in A-B$ there is the only such a pair.
Thus $a_x=a_x'$, $b_x=b_x'$ and we see from (\ref{tmp:14.11.2015_1}) that $c=c'$ as required.

{\it The second proof.}
Now let us prove (\ref{f:triangle_my}).
We have
$$
    |A\times B - \Delta (C)| = \sum_{q\in A-B} |B \cap (A-q) - C| \ge |A-B| |C| \,.
$$
The inequality above is trivial and the identity follows by the projection of points $(x,y) \in A\times B - \Delta (C)$, $(x,y) = (a-c,b-c)$,
$a\in A$, $b\in B$, $c\in C$ onto lines $q:=x-y=a-b \in A-B$.
If $q$ is fixed we see that the result of the projection is the intersection of the line $q=x-y$ with our set and moreover the ordinates of the points from the intersection belongs to $B \cap (A-q) - C$.
It is easy to check that the converse is also true.
This concludes the proof.
$\hfill\Box$
\end{proof}

\begin{exc}
    Using Ruzsa's triangle inequality, prove Freiman and Pigaev's result $|A+A|^{3/4} \le |A-A| \le |A+A|^{4/3}$.
\end{exc}

%Next results
The next theorem
provides some basic relations between the sizes of the higher
dimensional sumsets.
The
%following theorem
result
generalizes
%the well--known
Ruzsa's
triangle inequality.
%\cite{ruzsa1,ruzsa2}.

\begin{theorem}
Let $k \gs 1$ be a positive integer, and
let $A_1,\dots,A_k,B$ be finite subsets of an abelian group $\Gr$.
Further, let $W,Y \subseteq \Gr^k$, and $X,Z \subseteq \Gr$.
Then
\begin{equation}\label{f:Ruzsa_triangle1}
    |W\m X| |Y-\Delta(Z)| \le |Y\m W \m Z - \Delta(X)| \,,
%    |B| |A_1 \m A_2 \m \dots \m A_{k-1} - \Delta(A_k)| \le |A_1 \m A_2 \m \dots \m A_k - \Delta(B)|
%        \le
\end{equation}
\begin{equation}\label{f:Ruzsa_triangle2}
    |A_1 \m  \dots \m A_k - \Delta(B)|
        \le
            |A_1  \m \dots \m A_{m} - \Delta(A_{m+1})| |A_{m+1} \m \dots \m A_{k} - \Delta(B)|
\end{equation}
for any $m\in [k]$.
Furthermore, we have
\begin{equation}\label{f:Ruzsa_triangle''}
   %|A_1  \m \dots \m A_{k-1} - \Delta(A_k)| = |A_1 \m A_2 \m \dots \m A_{k-2} \m A_{k} - \Delta(A_{k-1})| \,.
    |Y \m Z - \Delta(X)| = |Y\m X - \Delta(Z)| \,.
\end{equation}
\label{t:Ruzsa_triangle}
\end{theorem}
\begin{proof}
To show the first
%two
inequality we apply Ruzsa's argument from the first proof of Lemma \ref{l:triangle_my}.
For every ${\bf a}\in Y - \Delta(Z)$
choose the smallest element (in any linear order of $Z$) $z \in Z$
such that ${\bf a}=(y_1-z,\dots,y_k-z)$
for some $(y_1,\dots,y_k) \in Y$.
Next, observe that the function
$$({\bf a},{\bf w},x)\mapsto (y_1-x,\dots,y_k-x,z-x,w_1-x,\dots,w_k-x) \,,$$
where ${\bf w} = (w_1,\dots,w_k) \in W$
from $(Y - \Delta(Z)) \times W \times X$ to $Y\m W \m Z - \Delta(X)$
is injective.
%The last inequality follows from Koester-Katz transform. Indeed, we
%have  $A+A_s\sbeq (A+A)_s,$ so that
%$$\sum_{\|s\|=k} |A+A_s|\ls \sum_{\|s\|=k} |(A+A)_s|=|A+A|^{l+1}\,.$$
%and the assertion follows.

    To obtain the second inequality consider the following matrix
\begin{displaymath}
    \mathbf{M} =
        \left( \begin{array}{cccccc}
                    1 & 0 & \ldots & 0 & 0 & -1 \\
                    0 & 1 & 0 & \ldots & 0 & -1 \\
                    0 & 0 & 1 & \ldots & 0 & -1 \\
                    \ldots & \ldots & \ldots & \ldots & \ldots \\
                    0 & \ldots & 0 & 0 & 1 & -1 \\
        \end{array} \right)
\end{displaymath}
Clearly, $A_1 \m \dots \m A_k - \Delta(B) = \mathbf{Im} ( \mathbf{M}|_{A_1 \m \dots \m A_k \m B})$.
Further, non--degenerate transformations of lines does not change the cardinality of the image.
Thus, subtracting the $(m+1)$th line, we obtain vectors of the form
$$
    (a_1-a_{m+1}, \dots, a_{m}-a_{m+1}, a_{m+1}- b, \dots, a_k - b)\,,
$$
which belong to
$
     (A_1  \m \dots \m A_{m} - \Delta(A_{m+1}) ) \m (A_{m+1} \m \dots \m A_{k} - \Delta(B) ) \,.
$
%where $a_j\in A_j$, $j\in [k]$, and $b\in B$ are  as before.

%Finally, inequality (\ref{f:Ruzsa_triangle'}) is trivial.
To obtain (\ref{f:Ruzsa_triangle''}) it is sufficient to show that
$$
   |Y \m Z - \Delta(X)| \le |Y\m X - \Delta(Z)| \,.
$$
%    |Y \m Z - \Delta(X)| \le |Y\m X - \Delta(Z)|
%    |A_1 \m A_2 \m \dots \m A_{k-1} - \Delta(A_k)| \le |A_1 \m A_2 \m \dots \m A_{k-2} \m A_{k} - \Delta(A_{k-1})| \,.
%$$
But the  map
$$
    (y_1-x,\dots,y_k-x,z-x) \mapsto (y_1-z,\dots,y_k-z,x-z) \,,
$$
where $(y_1,\dots,y_k) \in Y$, $x\in X$, $z\in Z$
is an injection.
%(another way to prove this is consider matrix $M$ and make a sequence of elementary transformations of lines of the matrix).
This completes the proof.
$\hfill\Box$
\end{proof}

\bigskip

There is another way to prove estimate (\ref{f:Ruzsa_triangle2}) in
spirit of Lemma 2.4 and Corollary 2.5 from \cite{SV}. We recall this
result, which follows from the definitions.
%We gather the results in the following proposition.

\begin{proposition}
    Let $k \ge 2$, $m \in [k]$ be positive integers, and
    let $A_1,\dots,A_k,B$ be finite subsets of an abelian group.
    Then
    \begin{equation}\label{f:characteristic1}
        A_1 \m \dots \m A_k - \Delta(B) = \{ (x_1,\dots,x_k) ~:~ B \cap (A_1-x_1) \cap \dots \cap(A_k - x_k) \neq \emptyset \}
    \end{equation}
    and
    \begin{equation}\label{f:characteristic2}
        A_1 \m \dots \m A_k - \Delta(B)
            =
    \end{equation}
    $$
                \bigcup_{(x_1,\dots,x_m) \in A_1 \m \dots \m A_m - \Delta(B)}
                    \{ (x_1,\dots,x_m) \} \m (A_{m+1} \m \dots \m A_k - \Delta(B \cap (A_1-x_1) \cap \dots \cap (A_m-x_m)) \,.
    $$
\label{p:characteristic}
\end{proposition}

Indeed, the intersection $B \cap (A_1-x_1) \cap \dots \cap(A_k - x_k)$ is nonempty iff, firstly, for  $(x_1,\dots,x_m)$ the intersection $\mathcal{B} := B \cap (A_1-x_1) \cap \dots \cap(A_m - x_m)$ is nonempty and, secondly, the intersection of the set $\mathcal{B}$ with $(A_{m+1}-x_{m+1}) \cap \dots \cap(A_k - x_k)$ is also nonempty.

\begin{corollary}
    We have
$$
    \sum_{s\in A-A} |A-A_s| = |A^2 - \D(A)| \,.
$$
\label{cor:A-A_s}
\end{corollary}

\begin{exc}
    Prove that
$$
    |A^2 + \D(A)| = \sum_{s\in A-A} |A+A_s| \ge |A| \cdot \max\{ |A+A|, |A-A| \} \,.
$$
    Moreover, let $n,m\ge 1$ be positive integers.
    Then
    \begin{equation}\label{tmp:12.05.2014_1}
        |A^{n+m} - \D(A)| \ge |A|^m |A^{n} - \D(A)| \,,
    \end{equation}
    and
    \begin{equation}\label{tmp:12.05.2014_2}
        |A^{n+m} + \D(A)| \ge |A|^m \max\{ |A^{n} + \D(A)|, |A^{n} - \D(A)| \} \,.
    \end{equation}
\end{exc}

\bigskip

%\noindent
From (\ref{f:characteristic1}) one can deduce another
characterization of the set
$A^k-\D( B)$, namely,
%$A\oplus_k A$ and $A\ominus_k A$:
%\begin{corollary}
 %%  Then
%    $$
%        A^k+\D( A) = \{ X\subseteq \Gr ~:~  |X| = k,\, A\not\sbeq (X - (\Gr\setminus A)) \, \} \,.
%    $$
%\end{corollary}
 $$
        A^k-\D( B) = \{ X\subseteq \Gr ~:~ |X| = k,\, B\not\subseteq  ( (\Gr\setminus A)-X)  \,\} \,.
    $$
Here we used $X$ to denote a multiset and a corresponding sequence
created from $X$.
Using the characterization it is easy to prove, that if  $A$ is a subset of finite abelian group $\Gr$ then there is $X$,
$|X| \sim \frac{N}{|A|} \cdot \log N$ such that $A+X = \Gr$.
Indeed, let $A^c = \Gr \setminus A$, and $k \sim \frac{N}{|A|} \cdot \log N$.
Consider
$$|(A^c)^k - \Delta(A^c)| \le |A^c|^{k+1} = N^{k+1} (1-|A|/N)^{k+1} < N^k \,.$$
Thus, there is a multiset $X$, $|X| = k$ such that $A^c \subseteq A-X$.
%Thus,
Whence
the set $-X\cup \{0\}$ has the required property.

\bigskip

%\noindent Proposition \ref{p:D_n}
The conception of the higher sumsets
allows us to introduce a hierarchy of basis
of abelian groups, i.e. of sets $B$ such that $B\pm B = \Gr$. For simplicity, if $B$ is a basis let us write
$B\oplus_kB$ and $B\ominus_k B$ for $B^k+\D(B)$ and $B^k-\D(B),$ respectively.

\begin{definition}
    Let $k\ge 1$ be a positive integer.
    A subset $B$ of an abelian group ${\bf G}$ is called {\it basis of depth $k$}
    if $B \ominus_k B= \Gr^k$.
\end{definition}

% It follows from Corollary \ref{c:A+B and E_k} that
    It follows from Theorem \ref{t:Ruzsa_triangle} that
 if   $B$ is  a basis of depth $k$ of finite abelian group $\Gr$,
    then for every set $A\subseteq \Gr$
\begin{equation}\label{f:diff-bases}
    |B + A| \ge |A|^{\frac{1}{k+1}} |\Gr|^{\frac{k}{k+1}} \,.
\end{equation}
Taking any one--element $A$ in formula (\ref{f:diff-bases}) we
obtain, in particular, that $|B| \ge |\Gr|^{\frac{k}{k+1}}$ for any
basis of depth $k$. It is easy to see, using Proposition
\ref{p:characteristic} that every  set  with $B$, $|B| > (1-1/(k+1))
|\Gr|$ is a basis of depth $k$ and this inequality is sharp.

If $S_1,\dots, S_k$ are any sets such that $S_1+\dots+S_k = \Gr$ then
the set $\bigcup_{j=1}^k (\sum_{i\neq j} (S_i - S_i))$ is a basis of depth $k$
(see Corollary \ref{c:G_bases} below, the construction can be found in \cite{Lev+_universal}).
Let us give another example.
%Furthermore, let us remark that
Using Weil's bounds for exponential sums
%({\bf Gauss sums })
we show that quadratic residuals in $\Z/p\Z$,
    for a prime $p,$ is a basis of depth $(\frac{1}{2}+o(1))\log p$.
Clearly, the bound is the best possible up to constants for subsets of $\Z/p\Z$ of the cardinality less than $p/2$.

\begin{proposition}
    Let $p$ be a prime number, and let $R$ be the set of quadratic residuals.
    Then $R$ is the bases of depth $k$, where $k 2^{k} < \sqrt{p}$.
\label{p:quadratic_residuals}
\end{proposition}
\begin{proof}
Clearly,
$$
    R(x) = \frac{1}{2} \left( \chi_0 (x) + \binom{x}{p} \right) \,,
$$
where $\binom{x}{p}$ is the Legendre symbol and $\chi_0 (x)$ is the main character.
Put $\a_0 = 0$.
For all distinct non--zero $\a_1,\dots,\a_k$, we have
\begin{eqnarray*}
    |R\cap (R-\a_1) \cap \dots \cap (R-\a_k)|
        &=&
            \frac{1}{2^k} \sum_x \prod_{j=0}^k \left( \chi_0 (x) + \binom{x+\a_j}{p} \right)
                \ge
                    \frac{1}{2^k} \left( p -1 - \sqrt{p} \cdot \sum_{j=2}^k j C^j_k
                    \right)\\
    &\ge&
        \frac{1}{2^k} \left( p - \sqrt{p} \cdot k 2^k \right) > 0 \,.
\end{eqnarray*}
We used the well--known Weil bound for exponential sums with
multiplicative characters (see e.g. \cite{Johnsen}). By formula
(\ref{f:characteristic1}) of  Proposition \ref{p:characteristic} we
see that $R\ominus_k R = \Z_p^k$. $\hfill\Box$
\end{proof}

\bigskip

Another consequence of Proposition \ref{p:quadratic_residuals}
is that quadratic non--residuals $Q$ (and, hence, quadratic residuals)
have no completion of size smaller then $(\frac12+o(1))\log p$,
that is a set $X$ such that $X+Q = \Z/p\Z$.

\bigskip

The next proposition is due to N.G. Moshchevitin.

\begin{proposition}
    Let $k_1,k_2$ be positive integers, and $X_1,\dots,X_{k_1},Y$, $Z_1,\dots,Z_{k_2},W$ be finite subsets of an abelian group.
    Then we have a bound
    $$
        |X_1 \m \dots \m X_{k_1} - \Delta(Y)| |Z_1 \m \dots \m Z_{k_2} - \Delta(W)|
            \le
    $$
    $$
            \le
                | (X_1 - W) \m \dots \m (X_{k_1}- W) \m (Y-Z_1) \m \dots \m (Y-Z_{k_2}) - \Delta(Y-W)| \,.
    $$
%    as well as
%    $$
%        \le | (X_1 - W) \m \dots (X_{k_1}- W) \m (Y-Z_1) \m \dots \m (Y-Z_{k_2}) - \Delta(Y-W)| \,.
%    $$
\end{proposition}
\begin{proof}
It is enough to observe that the map
$$
    (x_1 - y,\dots,x_{k_1}-y, z_1-w, \dots, z_{k_2} - w)
        \mapsto
$$
$$
        \mapsto
            (x_1-w-(y-w),\dots, x_{k_1} - w - (y-w), y-z_1 - (y-w), \dots, y-z_{k_2} - (y-w))
$$
where $x_j\in X_j$, $j\in [k_1]$, $y\in Y$, $z_j\in Z_j$, $j\in [k_2]$, $w\in W$ is  injective.
$\hfill\Box$
\end{proof}

\bigskip
\noindent In particular, the difference and the sum of two bases
of depths $k_1$ and $k_2$ is a basis of depth $k_1+k_2$. Let us also
formulate
%one more
a simple identity, which is a consequence
%It is easy to see that
%The proof coincide with the proof of the second part
of Theorem \ref{t:Ruzsa_triangle}.

\begin{corollary}
    Let $k \ge 2$ be a positive integer, and
    let $A_1,\dots,A_k$ be a subsets of a finite abelian group $\Gr$.
    Then
    \begin{equation}\label{}
        |A_1  \m \dots \m A_k - \Delta(\Gr)| = |\Gr| |A_1  \m \dots \m A_{k-1} - \Delta(A_k)| \,.
    \end{equation}
\label{c:G_bases}
\end{corollary}

\noindent Thus, $B$ is a basis of depth $k$ iff $B$ is
$(k+1)$--universal set (see \cite{abs}), i.e. a set that is for any
$x_1,\dots,x_{k+1} \in \Gr$ there is $z\in \Gr$ such that
$z+x_1,\dots, z+x_{k+1} \in B$.
A series of very interesting  examples of universal sets can be found in \cite{Lev+_universal}.

\bigskip

Let $A,B\subseteq \Gr$ be two finite sets. The magnification ratio
$R_B [A]$ of the pair $(A,B)$ (see e.g. \cite{tv}) is defined by
\begin{equation}\label{f:R_B[A]}
    R_B [A] = \min_{\emptyset \neq Z \subseteq A} \frac{|B+Z|}{|Z|} \,.
\end{equation}
We simply write $R[A]$ for $R_A [A]$. Petridis  \cite{p} obtained an
amazingly short proof of the following fundamental theorem, see book  \cite{Ruzsa_book}.

\begin{theorem}
    Let $A\subseteq \Gr$ be a finite set, and  $n,m$ be positive integers.
    Then
    $$
        |nA-mA| \le R^{n+m} [A] \cdot |A| \,.
    $$
\label{t:Petridis}
\end{theorem}

\no Another beautiful result (which implies Theorem
\ref{t:Petridis}) was proven also by Petridis \cite{p}.
%, see also book  \cite{Ruzsa_book}.

\begin{theorem}
    For any $A,B,C$, we have
    $$
        |B+C+X| \le R_{B} [A] \cdot |C+X| \,,
    $$
    where $X\subseteq A$ and $|B+X| = R_B [A] |X|$.
\label{t:Petridis_C}
\end{theorem}

\no  For a set  $B\subseteq \Gr^k$
    define
\begin{equation*}\label{}
        R_B [A] = \min_{\emptyset \neq Z \subseteq A} \frac{|B+\Delta(Z)|}{|Z|} \,.
\end{equation*}
In the next two results we assume that $X\subseteq A$ is such that
$|B+\D(X)|=R_B[A]|X|.$ It is easy to see that Petridis argument can
be adopted to higher dimensional sumsets, giving a  generalization
of Theorem \ref{t:Petridis_C}.

\begin{theorem}
    Let $A\subseteq \Gr$ and $B\subseteq \Gr^k.$
    Then for any $C\subseteq \Gr$, we have
    $$
        |B+\Delta(C+X)| \le R_{B} [A] \cdot |C+X| \,.
    $$
    \label{t:Petridis_C_Delta}
\end{theorem}

\no A consequence of Theorem \ref{t:Petridis_C_Delta}, we obtain a
generalization of the sum version of the triangle inequality
(see, e.g. \cite{Ruzsa_book}).
%(see, e.g. \cite{cochraine_comment}).

\begin{corollary}
    Let $k$ be a positive integer, $A,C\subseteq \Gr$ and $B\subseteq \Gr^k$ be finite sets.
    Then
    $$
        |A| |B+\Delta(C)| \le |B+\Delta(A)| |A+C| \,.
    $$
\label{c:triangle_plus}
\end{corollary}
\begin{proof}
 Using Theorem \ref{t:Petridis_C_Delta}, we have
$$
    |B+\Delta(C)| \le |B+\Delta(C+X)| \le R_B [A] \cdot |C+X| \le \frac{|B+\Delta(A)|}{|A|} |A+C|
$$
and the result follows.
$\hfill\Box$
\end{proof}

\bigskip

\no Thus, we have  the following sum--bases  analog of inequality (\ref{f:diff-bases}).
%Corollary \ref{c:diff-basis}.

\begin{corollary}
    Let $k$ be a positive integer, and $B\oplus_k B = \Gr^k$.
    Then for any set $A\subseteq \Gr$, we have
$$
    |B + A| \ge |A|^{\frac{1}{k+1}} |\Gr|^{\frac{k}{k+1}} \,.
$$
\label{c:sum-basis}
\end{corollary}

\bigskip

It is well known that Croot--Sisask \cite{cs} Lemma on almost periodicity of convolutions has become a central tool of Additive Combinatorics, see applications, say, in \cite{cs}, \cite{Schoen_Freiman}, \cite{Sanders_2A-2A}, \cite{Sanders_AP3_log}.
We conclude the section, showing that some  large subsets of $A_{\v{s}}- a_{\v{s}}$, $a\in A_{\v{s}}$ for typical $\v{s}$ are the sets of almost periods which appear in the
%proof
arguments
of Croot and Sisask.
It demonstrates the importance of the higher sumsets once more time.
For simplicity we consider just a symmetric case.

\begin{theorem}
    Let $\epsilon \in (0,1)$, $K\ge 1$ be real numbers and $p$ be a positive integer.
    Let also $A\subseteq \Gr$ be sets
    %such that
    with
    $|A-A| \le K|A|$ and let $f\in L_p (\Gr)$  be an arbitrary function.
    Then there is $a\in A$ and a set $T\subseteq A$, $|T| \ge |A| (2K)^{-O(\epsilon^{-2} p)}$
    such that for all $t\in T-a$ one has
\begin{equation}\label{f:CS}
    \| (f* A) (x+t) - (f* A) (x) \|_{L_p (\Gr,\,x)} \le \epsilon \| f\|_{L_p (\Gr)} \cdot |A|^{1/p} \,.
\end{equation}
\label{t:CS}
\end{theorem}
\begin{proof}
    Write $\mu_A (x) = A(x) / |A|$.
    Let  $k$ be a natural parameter, $k=O(\epsilon^{-2} p)$.
    Take $k$ points $x_1,\dots,x_k \in A$ uniformly and random and put
    $X_j (y) = f(y+x_j) - (f * \mu_A) (y)$.
    For any fixed $y$ the random variables $X_j (y)$ are independent with zero mean and variance at most
    $(|f|^2 * \mu_A) (y)$.
    By the Khinchin inequality for sums of independent random variables, we get
$$
    \| \sum_{j=1}^k X_j (y) \|_{L_p (\mu^k_A)} \ll (pk (|f|^2 * \mu_A) (y))^{1/2} \,.
$$
     Taking $p$th power, dividing by $k^p$, integrating over  $y$ and using H\"{o}lder inequality for $L_p(y-A)$, which gives us $(|f|^2 * \mu_A)^{p/2} (y) \le (|f|^p * \mu_A) (y)$, we obtain
$$
    \int \int \left| \frac{1}{k} \sum_{j=1}^k f(y+x_j) - (f * \mu_A) (y) \right|^p dy\, d\mu_A^k (\v{x})
        \ll
            (pk^{-1} \| f \|^2_{L_p(\Gr)})^{p/2} \,.
$$
    Here $\v{x} = (x_1,\dots,x_k)$.
    Applying the H\"{o}lder inequality again and recalling than $k=O(\epsilon^{-2} p)$,
    we have
\begin{equation}\label{tmp:25.07.2014_1}
    \int \| \frac{1}{k} \sum_{j=1}^k f(y+x_j) - (f * \mu_A) (y) \|_{L_p(\Gr,\,y)} \,\,d\mu_A^k (\v{x})
        \le
            \epsilon \| f \|_{L_p(\Gr)} / 4\,.
\end{equation}
    From estimate (\ref{tmp:25.07.2014_1}) it follows that the set $L$ of all $\v{x} \in A^k$ such that the norm in the inequality less than $\epsilon \| f \|_{L_p(\Gr)} / 2$ has measure  $\mu_A^k (L) \ge 1/2$.
    Putting $\mu_{\v{x}}$ equals the probability measure sitting on the points $x_1,\dots,x_k$, we obtain
    $(\mu_{\v{x}} *f) (y) = \frac{1}{k} \sum_{j=1}^k f(y+x_j)$ and then (\ref{tmp:25.07.2014_1}) says us % that
\begin{equation}\label{tmp:29.11.2015_1}
     \| \mu_{\v{x}} (y) - (f * \mu_A) (y) \|_{L_p(\Gr,\,y)} \le \epsilon \| f \|_{L_p(\Gr)} / 2
\end{equation}
    provided by $\v{x} \in L$.

    Now we can construct our set of almost periods.
    Clearly,
$$
    A_{\v{s}} = \{ a\in A ~:~ \D(a) + \v{s} \in A^k \} \,, \quad \quad \v{s} \in A^k - \D(A) \,.
$$
    Thus, put
$$
    A'_{\v{s}} = \{ a\in A ~:~ \D(a) + \v{s} \in L \}  \subseteq A_{\v{s}} \,.
$$
    We claim that any set $A'_{\v{s}}$, $\v{s} \in L$ is a set of almost periods (it corresponds to the arguments of T. Sanders from \cite{Sanders_2A-2A}, say).
    Indeed, by the definition of the set $L$, see formula (\ref{tmp:29.11.2015_1}),
    we have for any $a\in A'_{\v{s}}$ that
\begin{equation}\label{tmp:01.12.2015_1}
     \| \mu_{\D(a)+\v{s}}  (y) - (f * \mu_A) (y) \|_{L_p(\Gr,\,y)}
        =
            \| \mu_{\v{s}}  (y+a) - (f * \mu_A) (y) \|_{L_p(\Gr,\,y)}
                \le
                    \epsilon \| f \|_{L_p(\Gr)} / 2
\end{equation}
and
\begin{equation}\label{tmp:01.12.2015_2}
    \| \mu_{\v{s}}  (y) - (f * \mu_A) (y) \|_{L_p(\Gr,\,y)}
        \le
                    \epsilon \| f \|_{L_p(\Gr)} / 2 \,.
\end{equation}
    Thus by the triangle inequality and because of any shift preserves $L_p$--norm, one has
$$
    \| (f * \mu_A) (y-a) - (f * \mu_A) (y) \|_{L_p(\Gr,\,y)}
        =
             \| (f * \mu_A) (y+a) - (f * \mu_A) (y) \|_{L_p(\Gr,\,y)}
        \le
                    \epsilon \| f \|_{L_p(\Gr)} \,.
$$

    Finally, we show that there is large set $A'_{\v{s}}$ with $\v{s} \in L$.
    Indeed, clearly, $A'_{\v{s}} (a) = A(a) L(\v{s} + \D(a))$ and hence
$$
    \sum_{\v{s} \in L} |A'_{\v{s}}| = \sum_a A(a) (L \c L) (\D(a)) = \sum_{\v{z}} \D (A) (\v{z}) (L \c L) (\v{z}) := \sigma \,.
$$
    By (\ref{f:E_sumsets_CS}), we have
$$
    \frac{|L|^2 |A|^2}{|A^k - \D(A)|} \le \frac{|L|^2 |A|^2}{|L - \D(A)|} \le \E (\D(A), L)
        =
            \sum_{\v{z}} (\D (A) \c \D(A)) (\v{z}) (L \c L) (\v{z})
                \le
                    |A| \sigma \,,
$$
    and thus
    there is $\v{s} \in L$ such that
\begin{equation}\label{tmp:01.12.2015_3}
    |A'_{\v{s}}| \ge |A| \cdot \frac{|L|}{|A^k - \D(A)|} \ge 2^{-1} |A| \cdot \frac{|A|^k}{|A^k - \D(A)|}
        \ge
            2^{-1} K^{-k} |A|
\end{equation}
    as required.

    The arguments above, actually, demonstrate that one can drop the requirement $\v{s} \in L$ (we thanks T. Schoen who show us the proof).
    Indeed, just take any $a\in A'_{\v{s}}$ for $\v{s} \in A^k - \D (A)$ and nonempty $A'_{\v{s}}$
    and after that check, using calculations in (\ref{tmp:01.12.2015_1}), (\ref{tmp:01.12.2015_2})  that the set
    $T:= A'_{\v{s}} - a$ is a set of almost periods.
    Moreover, the choice of $T$ allows us to find large $A'_{\v{s}}$ even simpler.
    Indeed, $\sum_{\v{s}} |A'_{\v{s}}| = |L| |A| \ge |A|^{k+1}/2$ and thus lower bound (\ref{tmp:01.12.2015_3}) holds immediately.
    This completes the proof.
$\hfill\Box$
\end{proof}

%see e.g. \cite{Sanders_survey}, \cite{Sanders_2A-2A}.
\section{Higher energies}
\label{sec:lecture2}

%\bigskip
%$\hfill\Box$

In the section  we  develop the functional point of view on the higher sumsets.
First of all let us define the {\it generalized convolutions.}
Let $k$ be a positive integer and $f_1, \dots, f_{k+1} : \Gr \to \C$ be any functions.
Denote by
$$ \Cf_{k+1} (f_1,\dots,f_{k+1}) (x_1,\dots, x_{k})$$
the function
$$
    %\Cf_k(F) (x) =
    \Cf_{k+1} (f_1,\dots,f_{k+1}) (x_1,\dots, x_{k}) = \sum_z f_1 (z) f_2 (z+x_1) \dots f_{k+1} (z+x_{k}) \,.
$$
Thus, $\Cf_2 (f_1,f_2) (x) = (f_1 \circ f_2) (x)$.
%Put $C_1 (f) = \sum_z f(z)$.
If $f_1=\dots=f_{k+1}=f$ then write
$\Cf_{k+1} (f) (x_1,\dots, x_{k})$ for $\Cf_{k+1} (f_1,\dots,f_{k+1}) (x_1,\dots, x_{k})$.
It is easy to see that
$$
    \supp \Cf_{k+1} (A) = A^k - \D(A) \,,
$$
that is a higher sumset from the previous section.
Hence  higher sumsets appear naturally as supports of these generalized convolutions.

Let us make a general remark about the functions $\Cf_{k+1} (f_1,\dots,f_{k+1}) (x_1,\dots, x_{k})$.
Suppose that $l,k \ge 2$ be positive integers and ${\bf F} = (f_{ij})$, $i=0,\dots,l-1;j=0,\dots,k-1$
be a functional matrix, $f_{ij} : \Gr \to \C$.
Let $R_0,\dots,R_{l-1}$ and $C_0,\dots,C_{k-1}$ be rows and columns of the matrix, correspondingly.
The following commutative relation holds.

\begin{lemma}
    For any positive integers $l,k \ge 2$, we have
    \begin{equation}\label{f:commutative_C}
        \Cf_l (\Cf_k (R_0), \dots, \Cf_k (R_{l-1})) = \Cf_k (\Cf_l (C_0), \dots, \Cf_l (C_{k-1})) \,.
    \end{equation}
\label{l:commutative_C}
\end{lemma}
\begin{proof}
Let $y^{(i)} = (y_{i1}, \dots,y_{i(k-1)})$, $i\in [l-1]$,
and $y_{(j)} = (y_{1j}, \dots,y_{(l-1)j})$, $j\in [k-1]$.
Put also $y_{0j} = 0$, $j=0,\dots,k-1$, $y_{i0}=0$, $i=1,\dots,l-1$ and  $x_0=0$.
We have
$$
    \Cf_l (\Cf_k (R_0), \dots, \Cf_k (R_{l-1})) (y^{(1)},\dots,y^{(l-1)})
        =
$$
$$
        =
            \sum_{x_1,\dots,x_{k-1}} \Cf_k (R_0) (x_1,\dots,x_{k-1}) \Cf_k (R_1) (x_1 + y_{11},\dots,x_{k-1} + y_{1(k-1)})
                \dots
$$
$$
                \dots
                    \Cf_k (R_{l-1}) (x_1 + y_{(l-1)1},\dots,x_{k-1} + y_{(l-1)(k-1)})
%        =
%$$
%$$
        =
            \sum_{x_0,\dots,x_{k-1}} \sum_{z_0,\dots,z_{l-1}} \prod_{i=0}^{l-1} \prod_{j=0}^{k-1} f_{ij}(x_j+y_{ij}+z_i) \,.
$$
Changing the summation, we obtain
$$
    \Cf_l (\Cf_k (R_0), \dots, \Cf_k (R_{l-1})) (y^{(1)},\dots,y^{(l-1)})
        =
$$
$$
        =
            \sum_{z_1,\dots,z_{l-1}} \Cf_l(C_0) (z_1,\dots,z_{l-1})
                \Cf_l(C_1) (z_1+y_{11},\dots,z_{l-1}+y_{(l-1)1}) \dots
$$
$$
    \dots
                \Cf_l(C_{l-1}) (z_1+y_{1(k-1)},\dots,z_{l-1}+y_{(l-1)(k-1)})
                    =
                        \Cf_k (\Cf_l (C_0), \dots, \Cf_l (C_{k-1})) (y_{(1)}, \dots, y_{(k-1)}) \,.
$$
as required.
$\hfill\Box$
\end{proof}

%The next corollary allows us to make simple analysis

\begin{corollary}
    For any functions the following holds
$$
    \sum_{x_1,\dots, x_{l-1}} \Cf_l (f_0,\dots,f_{l-1}) (x_1,\dots, x_{l-1})\, \Cf_l (g_0,\dots,g_{l-1}) (x_1,\dots, x_{l-1})
        =
$$
\begin{equation}\label{f:scalar_C}
        =
        \sum_z (f_0 \circ g_0) (z) \dots (f_{l-1} \circ g_{l-1}) (z) \quad \quad  \mbox{\bf (scalar product), }
\end{equation}
moreover
$$
    \sum_{x_1,\dots, x_{l-1}} \Cf_l (f_0) (x_1,\dots, x_{l-1}) \dots  \Cf_l (f_{k-1}) (x_1,\dots, x_{l-1})
        =
$$
\begin{equation}\label{f:gen_C}
        =
            \sum_{y_1,\dots,y_{k-1}} \Cf^l_k (f_0,\dots,f_{k-1}) (y_1,\dots,y_{k-1})
                \quad \quad  \mbox{\bf (multi--scalar product), }
\end{equation}
    and
%moreover
$$
    \sum_{x_1,\dots, x_{l-1}} \Cf_l (f_0) (x_1,\dots, x_{l-1})\, (\Cf_l (f_1) \circ \dots \circ \Cf_l (f_{k-1})) (x_1,\dots, x_{l-1})
        =
$$
\begin{equation}\label{f:conv_C}
        =
            \sum_{z} (f_0 \circ \dots \circ f_{k-1})^l (z)
                \quad \quad  \mbox{\bf (} \sigma_{k} \quad \mbox{\bf for } \quad \Cf_l \mbox{\bf )}  \,.
\end{equation}
\label{c:commutative_C}
\end{corollary}
\begin{proof}
Take $k=2$ in (\ref{f:commutative_C}).
Thus ${\bf F}$ is a $l\times 2$ matrix in the case.
We have
$$
    \Cf_l (f_0 \circ g_0,\dots,f_{l-1} \circ g_{l-1}) (x_1,\dots, x_{l-1})
        =
            (\Cf_l (f_0,\dots,f_{l-1}) \circ \Cf_l (g_0,\dots,g_{l-1})) (x_1,\dots, x_{l-1}) \,.
$$
Putting $x_j=0$, $j\in [l-1]$, we obtain (\ref{f:scalar_C}).
Applying the last formula $(k-2)$ times and after that formula (\ref{f:scalar_C}), we get (\ref{f:conv_C}).
Finally, taking ${\bf F}_{ij} = f_j$, $i=0,\dots,l-1;j=0,\dots,k-1$
and putting all variables in (\ref{f:commutative_C}) equal zero, we obtain (\ref{f:gen_C}).
This completes the proof.
$\hfill\Box$
\end{proof}

\bigskip

Now we can make the main definition of the section.

%Let us

\begin{definition}
    Let  $k,l$ be any positive integers, and $A\subseteq \Gr$ be a set.
    Then
$$
    \E_{k,l} (A) := \sum_{x_1, \dots, x_{k-1}} \Cf^l_{k} (A) (x_1, \dots, x_{k-1}) \,.
$$
    If $k=2$ then we write $\E_l (A)$ for $\E_{2,l} (A)$.
    Clearly,
$$
    \E_{k+1} (A) = \E(\D_k (A), A^k) \,.
$$
\end{definition}

    If $l=1$ then, obviously, $\E_{k,1} (A) = |A|^{k+1}$ and if $k=1$ then put for symmetry  $\E_{1,l} (A) = |A|^{l+1}$.
    Thus, the cardinality of a set can be considered as a degenerate energy.
    Sometimes we need in energies $\E_\a (A)$ for real $\a$, that is $\E_\a (A) = \sum_x (A \c A)^\a (x)$.

\bigskip

\begin{corollary}
    For any positive integers $k,l$ one has $\E_{k,l} (A) = \E_{l,k} (A)$.
    In particular
\begin{equation}\label{f:symmetry_E}
    \sum_{(\a,\beta) \in A^2 - \D(A)} \Cf_3^2 (\a,\beta) = \E_{3,2} (A) = \E_{2,3} (A) = \E_3 (A) \,.
\end{equation}
\label{c:symmetry_E}
\end{corollary}

\bigskip

Now let us prove a lemma on the connection of the higher energies of the sets $A_{\v{s}}$ and $A$.

\begin{lemma}
    Let $A\subseteq \Gr$ be a set, $l\ge 1$, $k\ge 2$ be positive integers.
    Then
    \begin{equation}\label{f:moments_of_gen_convolutions}
        \sum_{\v{s}_1,\dots, \v{s}_k} \sum_{z_1,\dots,z_{k-1}} \Cf^l_k (A_{\v{s}_1}, \dots, A_{\v{s}_k}) (z_1,\dots,z_{k-1})
            =
                \sum_{x_1,\dots,x_{l-1}} \Cf^{\| \v{s} \|+k}_l (A) (x_1,\dots,x_{l-1}) \,,
    \end{equation}
    where $\| \v{s} \| = \sum_{j=1}^k |\v{s}_j|$.
    In particular,
    \begin{equation}\label{f:gen_conv1}
        \sum_{\v{s}_1,\dots, \v{s}_k} \sum_{z_1,\dots,z_{k-1}} \Cf_k (A_{\v{s}_1}, \dots, A_{\v{s}_k}) (z_1,\dots,z_{k-1})
            = |A|^{\| \v{s} \| + k} \,,
    \end{equation}
    and
    \begin{equation}\label{f:gen_conv2}
        \sum_{\v{s}_1,\dots, \v{s}_k} \sum_{z_1,\dots,z_{k-1}} \Cf^2_k (A_{\v{s}_1}, \dots, A_{\v{s}_k}) (z_1,\dots,z_{k-1})
            =
                \sum_{\v{s}_1,\dots, \v{s}_k} \E_k (A_{\v{s}_1}, \dots, A_{\v{s}_k}) = \E_{\| \v{s} \| + k} (A) \,.
    \end{equation}
\label{l:moments_of_gen_convolutions}
\end{lemma}
\begin{proof}
Let us put $z_0=0$ for convenience.
We have
\begin{equation}\label{f:tmp_17_12_2010_1}
    \sum_{\v{s}_1,\dots, \v{s}_k} \sum_{z_1,\dots,z_{k-1}} \Cf^l_k (A_{\v{s}_1}, \dots, A_{\v{s}_k}) (z_1,\dots,z_{k-1})
        =
            \sum_{\v{s}_1,\dots, \v{s}_k} \sum_{z_1,\dots,z_{k-1}}
                \sum_{w_1,\dots,w_l} \prod_{j=1}^l \prod_{i=1}^k A_{\v{s}_i} (w_j+z_{i-1})
%                    =
\end{equation}
$$
    =
        \sum_{w_1,\dots,w_l} \sum_{z_1,\dots,z_{k-1}} \Cf^{\| \v{s} \|}_k (A) (w_2-w_1,\dots,w_l-w_1)
            \prod_{j=1}^l \prod_{i=1}^k A (w_j+z_{i-1})
            =
$$
$$
            =
            \sum_{w_1,\dots,w_l} \Cf^{\| \v{s} \|+k-1}_k (A) (w_2-w_1,\dots,w_l-w_1) A(w_1) \dots A(w_l)
                =
        \sum_{x_1,\dots,x_{l-1}} \Cf^{\| \v{s} \|+k}_l (A) (x_1,\dots,x_{l-1})
        \,,
$$
because each component of any vector $\v{s}_i$ appears at formula (\ref{f:tmp_17_12_2010_1}) exactly $l$ times.
This completes the proof.
$\hfill\Box$
\end{proof}

\begin{corollary}
    For any $A\subseteq \Gr$, we have
\begin{equation}\label{f:E_3,E_4_A_s}
    \sum_s \E (A,A_s) = \E_3 (A)\,, \quad \quad  \sum_{s,t} \E (A_s,A_t) = \E_4 (A)\,.
\end{equation}
\label{c:E_3,E_4_A_s}
\end{corollary}

\bigskip

Now let us look at the higher energies from the point of view of Fourier analysis.
First of all, let us define a classical generalization of the additive energy of a set, see e.g. \cite{K_Tula}.
Let
$$
   \T_k (A) := | \{ a_1 + \dots + a_k = a'_1 + \dots + a'_k  ~:~ a_1, \dots, a_k, a'_1,\dots,a'_k \in A \} | \,.
        %\,.
$$
Using formula (\ref{svertka}) it can be shown that
$$
    \T_k (A) = \frac{1}{N} \sum_\xi |\FF{A} (\xi)|^{2k} \,.
$$
Let also
$$
    \sigma_k (A) := (A*_k A)(0)=| \{ a_1 + \dots + a_k = 0 ~:~ a_1, \dots, a_k \in A \} | \,.
$$
%Notice that for a symmetric set $A$ that is $A=-A$ one has $\sigma_2
%(A) = |A|$ and $\sigma_{2k} (A) = \T_k (A)$.
%Having a set $P\subseteq A-A$ we write $\sigma_P (A) := \sum_{x\in P} (A\c A) (x)$.

Quantities $\E_k (A)$ and $\T_k (A)$ are "dual"\, in some sense.
For example in \cite{s_mixed}, Note 6.6 (see also \cite{SS1}) it was proved that
\begin{equation}\label{f:uncerteinty}
    \left( \frac{\E_{3/2} (A)}{|A|} \right)^{2k}
        \le
            \E_k (A) \T_k (A) \,,
\end{equation}
provided by $k$ is even.
Moreover, from (\ref{F:Fourier})---(\ref{f:inverse}), (\ref{f:char_char}) it follows that
%$\E_{2k} (\FF{A}, \dots, \FF{A}, \ov{\FF{A}}, \dots, \ov{\FF{A}})
\begin{equation}\label{f:ET_1}
    \t{\E}_{2k} (\FF{A}) := \sum_{x} ( \ov{\FF{A}} \c \FF{A})^{k} (x) (\FF{A} \c \ov{\FF{A}})^{k} (x)
        = N^{2k+1} \T_k (A) \,,
\end{equation}
and
\begin{equation}\label{f:ET_2}
    \T_k (|\FF{A}|^2) = N^{2k-1} \E_{2k} (A) \,.
\end{equation}

\bigskip

Another dual formulae can be find in \cite{SS1}.
We give just an example.

\begin{exc}
Let $A$ be a subset of an abelian group.
Then for every $k\in \N$, we have
\begin{equation}\label{f:E_k_&_sigma_k}
    |A|^{2k} \le \E_k (A) \cdot \sigma_k (A-A) \,, \quad \quad |A|^{4k} \le \E_{2k} (A) \cdot \T_{k} (A+A) \,,
\end{equation}
and
\begin{equation}\label{f:E_k_&_E_k}
    |A|^{2k+4} \le \E_{k+2} (A) \cdot \E_k (A-A) \,, \quad \quad |A|^{2k+4} \le \E_{k+2} (A) \cdot \E_k (A+A) \,.
\end{equation}
\end{exc}

\bigskip

Now we are ready to obtain a first application of the method of higher energies.

Let $A=\{a_1,\dots,a_n\},\, a_i<a_{i+1}$ be  a set of real numbers.
We say that $A$ is \emph{convex} if
$$a_{i+1}-a_i>a_i-a_{i-1}$$ for every $i=2,\dots,n-1.$
Hegyv\'ari \cite{h}, answering a question of Erd{\H o}s, proved that if $A$ is convex then
$$|A+A|\gg |A|\log |A|/\log\log |A|\,.$$ This result was later improved by many authors. Konyagin \cite{k} and  Garaev \cite{g1} showed
independently that
the additive energy $\E(A)=\E(A,A)$
of a convex set is $\ll |A|^{5/2},$ which immediately implies that
$$|A\pm A|\gg |A|^{3/2}\,.$$
Elekes, Nathanson and Ruzsa \cite{enr} proved that if $A$ is convex then
$$|A+B|\gg |A|^{3/2}$$
for every set $B$ with $|B|=|A|$.
Finally, Solymosi \cite{soly} generalized the above inequality, showing that if $A$ is a set with distinct consecutive differences i.e.
$a_{i+1}-a_{i}=a_{j+1}-a_{j}$ implies $i=j$ then
\begin{equation}\label{isoly}
|A+B|\gg |A||B|^{1/2}
\end{equation}
for every set $B.$

\bigskip

%The aim of this note is to establish
In \cite{SS3} the following theorem was proved.

\begin{theorem}\label{thm:main} Let $A$ be a convex set. Then
\begin{equation}\label{f:convex1}
|A-A|\gg |A|^{8/5}\log^{-2/5} |A|
\,,
\end{equation}
and
\begin{equation}\label{f:convex2}
|A+A|\gg |A|^{14/9}\log^{-2/3} |A| \,, \quad \and
    |A+A|^3 |A-A|^2 \log^{2} |A| \gg |A|^8 \,.
\end{equation}
\end{theorem}

For simplicity we have deal just with the difference case, so we will prove estimate (\ref{f:convex1}).
The best result for the  {\it sumsets} of convex sets can be found in \cite{s_ST}.

We need in several lemmas.
%results.
The first  one is
%a consequence of
the Szemer\'edi--Trotter theorem
\cite{sz-t}, see also \cite{tv}. We call a set $\mathcal{L}$ of continuous
plane curves a {\it pseudo-line system} if any two members of $\mathcal{L}$
share at most one point in common.
Define the {\it number of indices} $\mathcal{I} (\mathcal{P},\mathcal{L})$ as
$\mathcal{I}(\mathcal{P},\mathcal{L})=|\{(p,l)\in \mathcal{P}\times \mathcal{L} : p\in l\}|$.

\begin{theorem}\label{t:SzT}
%(\cite{sz-t})
Let $\mathcal{P}$ be a set of points and let $\mathcal{L}$ be a pseudo-line system.
Then
$$\mathcal{I}(\mathcal{P},\mathcal{L}) \ll |\mathcal{P}|^{2/3}|\mathcal{L}|^{2/3}+|\mathcal{P}|+|\mathcal{L}|\,.$$
\end{theorem}

%Using theorem above we can prove

The next definition is basically from \cite{enr}.

\begin{definition}
    Let $A \subset \R$ be a finite set.
    Put
\begin{equation}\label{f:d_general}
    d(A) = \inf_{f}\,
        %\min_{a\in \R}\,
        \min_{C \neq \emptyset}\, \frac{|f(A)+C|^2}{|A| |C|} \,,
\end{equation}
    where the
    %minimum
    infinum
    over $f$ is taken over all convex/concave functions.
\end{definition}
It is easy to see that $1\le d(A) \le |A|$.

\bigskip

\begin{exc}
    For simplicity fix $f$ in (\ref{f:d_general}) and show that the minimum is attained.
\end{exc}

\bigskip

We need in the following lemma, see \cite{RR-NS}, previous results of similar form were proved in \cite{Li}, \cite{SS1}.

\begin{lemma}
    Let $A, B \subset \R$ be finite sets.
    Then for all $\tau > 0$ one has
\begin{equation}\label{f:convolutions_d_gen}
    |\{ x ~:~ (A\c B) (x) \ge \tau \}| \ll d(A) \cdot \frac{|A||B|^2}{\tau^3} \,.
\end{equation}
\label{l:convolutions_d_gen}
\end{lemma}
\begin{proof}
    Without loosing of generality we can suppose that $\tau$ is a positive integer.
    Let $C$ be a set, where the minimum in (\ref{f:d_general}) is attained.
    Take $x\in \R$ such that $(A\c B) (x) \ge \tau$.
    Then $x=a_1-b_1 = \dots = a_\tau - b_\tau$ for some $a_j \in A$, $b_j \in B$.
    Hence for any $c\in C$, we have
\begin{equation}\label{f:c_expression_f}
    c = -f(x+b_j) + f(a_j) + c \,, \quad \quad j\in [\tau] \,.
\end{equation}
    Consider the family of convex curves $\mathcal{L} = \{ l_{s,b} \}$, where $s\in f(A)+C$, $b\in B$, defining by the equation
    $$l_{s,b} = \{ (x,y) ~:~ y=-f(x+b)+s \} \,.$$
    Clearly, $|\mathcal{L}| = |f(A)+C| |B|$.
    Let also $\mathcal{P}$ be the set of all intersecting points, defining by the curves.
    %namely, $\mathcal{P} = $
    In the terms identity (\ref{f:c_expression_f}) says us that the point $(x,c)$ belongs to the $\tau$ curves
    $l_{f(a_j) + c, b_j}$, $j\in [\tau]$.
    Hence $(x,c) \in \mathcal{P}_t$.
    The set of curves $\mathcal{L}$ satisfies the conditions of the Szemer\'{e}di--Trotter Theorem.
    Indeed, if $(\a,\beta) \in l_{s,b} \cap l_{s',b'}$ then $h(x) := f(x+b')-f(x+b)+s-s'=0$.
    By the assumption $f(x)$ is a convex/concave function.
    Hence $h(x)$ is a monotone function.
    It follows that the equation $h(x)=0$ has at most one solution.
    Thus by Theorem \ref{t:SzT}, we have
    $$
        |C| \cdot | \{ x ~:~ (A\c B) (x) \ge \tau \}| \le |\mathcal{P}_t|
        \ll
            \frac{|f(A)+C|^2 |B|^2}{\tau^3} + \frac{|f(A)+C| |B|}{\tau}
                \ll
    $$
    $$
                \ll
            \frac{|f(A)+C|^2 |B|^2}{\tau^3} \,.
    $$
    In the last formula, we have used a trivial inequality
$$
    \tau^2 \le (\min\{ |A|, |B|\})^2 \le |A| |B| \le |f(A)+C| |B| \,.
$$
    This completes the proof.
$\hfill\Box$
\end{proof}

\begin{exc}
    Let $A$ be a convex set and $A' \subseteq A$ be its subset such that $|A'| \gg |A|$.
    Then show that $d(A') \ll 1$.
    In particular it implies that for an arbitrary $B$ the following holds
\begin{equation}\label{f:A'_B}
    |A'+B| \gg |A| |B|^{1/2} \,.
\end{equation}
\label{ex:A'_B}
\end{exc}

\bigskip

Let $A$ be a convex set and $B$ be an arbitrary set.
Order the elements  $s\in B-A$  such that
$(A \c B)(s_1)\gs (A\c B)(s_2)\gs\dots \gs (A \c B) (s_t),\, t=|B-A|.$
The next lemma was
proved  in \cite{g1}, say,  and is immediate consequence of  Lemma \ref{l:convolutions_d_gen}.
% and in \cite{ik}.

\begin{lemma}\label{lcon}
Let $A$ be a convex set and $B$ be an arbitrary set.
Then for every $j\gs 1$ we have
$$(A\c B) (s_j)\ll (|A||B|^2)^{1/3} j^{-1/3}\,.$$
\end{lemma}

\begin{corollary}\label{core3}
Let $A$ be a convex set.
Then
$$\E_3(A)\ll |A|^3\log |A|\,,$$
and
$$
    \E(A,B) \ll |A| |B|^{3/2} \,.
$$
\end{corollary}
\begin{proof}
    To obtain the first formula, just use Lemma \ref{lcon} with $B=A$
$$
    \E_3(A) = \sum_x (A\c A)^3 (x) = \sum_{j\ge 1} (A\c A)^3 (s_j) \ll |A|^3 \sum_{j=1}^{|A|} j^{-1}
        \ll
            |A|^3 \log |A| \,.
$$
    To get the second estimate, choose a parameter $\tau = |B|^{1/2}$ and use Lemma \ref{lcon} again
$$
    \E(A,B) = \sum_x (A\c B)^2 (x)
        =
            \sum_{x ~:~ (A\c B) (x) < \tau} (A\c B)^2 (x) + \sum_{x ~:~ (A\c B) (x) \ge \tau} (A\c B)^2 (x)
                \le
$$
$$
                \le
                    \tau |A| |B| +  \sum_{x ~:~ (A\c B) (x) \ge \tau} (A\c B)^2 (x)
                        \le
                             \tau |A| |B| + \sum_{j \ge 1 ~:~ (A\c B) (s_j) \ge \tau} (A\c B)^2 (s_j)
                                \ll
$$
$$
    \ll
        \tau |A| |B| + (|A| |B|^2)^{2/3} \sum_{j=1}^{|A| |B|^2 \tau^{-3}} j^{-2/3}
            \ll
                \tau |A| |B| + \frac{|A| |B|^2}{\tau}
                    \ll
                        |A| |B|^{3/2} \,.
$$
    This completes the proof.
$\hfill\Box$
\end{proof}

\bigskip

Using the higher energies we obtain a result from \cite{SS3}.

\begin{theorem}
    Let $A$ be a convex set. Then
\begin{equation}\label{f:convex1+}
|A-A|\gg |A|^{8/5}\log^{-2/5} |A| \,.
\end{equation}
\end{theorem}
\begin{proof}
    Put $D=A-A$.
    We have
\begin{equation}\label{tmp:29.11.2015_2}
    |A|^2 = \sum_{s\in D} |A_s| \le 2 \sum_{s\in P} |A_s| \,,
\end{equation}
    where $P := \{ s\in D ~:~ |A_s| \ge |A|^2 /(2|D|) \}$.
    Thus, using (\ref{tmp:29.11.2015_2}) and the Cauchy--Schwarz inequality, we obtain
$$
    2^{-1} |A|^3 \le |A| \sum_{s\in P} |A_s| = \sum_{s\in P} \sum_{x} (A\c A_s) (x)
        \le
            \sum_{s\in P} \E^{1/2} (A,A_s) |A-A_s|^{1/2} \,.
$$
    Applying the Cauchy--Schwarz inequality once more time, we get in view of Corollary \ref{c:E_3,E_4_A_s} that
$$
    |A|^6 \ll \sum_{s} \E (A,A_s) \cdot \sum_{s\in P} |A-A_s| = \E_3 (A)  \cdot  \sum_{s\in P} |A-A_s| \,.
$$
    By Katz--Koester trick \cite{kk}, that is inclusion $A-A_s \subseteq D \cap (D+s)$, we have
$$
    |A|^6 \ll \E_3 (A) \cdot  \sum_{s\in P} |D_s| \,.
$$
    But by the definition of the set $P$, we know that $|A_s| \ge |A|^2 /(2|D|)$ for all  $s\in P$.
    Hence
\begin{equation}\label{f:D_const}
    |A|^8 \ll |D| \E_3 (A)  \cdot \sum_{s} |A_s| |D_s| = |D| \E_3 (A)  \E(A,D) \,.
\end{equation}
    Finally, applying  Corollary \ref{core3},
%    we obtain
    we get
$$
    |A|^8 \ll |D| |A|^3 \log |A| \cdot |A| |D|^{3/2} \,.
$$
    Using some algebra, we obtain the result.
    This completes the proof.
$\hfill\Box$
\end{proof}

\bigskip

In the next section we will see that inequality (\ref{f:D_const}) is just a consequence of a simple fact about some specific operators.

\section{Higher energies and eigenvalues of some operators}
\label{sec:lecture3}

%\bigskip
%$\hfill\Box$

Now we introduce some operators, which firstly appeared in  \cite{s} in a dual form and were connected with some restrictions problems of Fourier analysis.
Nevertheless, our definition below does not use any Fourier analysis and is a purely combinatorial.

Let $g : \Gr \to \C$ be a function, and $A\subseteq \Gr$ be a finite set.
By $\oT^{g}_A$ denote the $(|A| \times |A|)$  matrix with indices in the set $A$
\begin{equation}\label{def:operator1}
    \oT^{g}_A (x,y) = g(x-y) A(x) A(y) \,,
\end{equation}
and, similarly, put
\begin{equation}\label{def:operator1}
    \t{\oT}^{g}_A (x,y) = g(x+y) A(x) A(y) \,.
\end{equation}
It is easy to see that $\oT^{g}_A$ is hermitian iff $\ov{g(-x)} = g(x)$ and $\t{\oT}^g_A$ is hermitian iff $g$ is a real function.
Basically, we shall deal with $\oT^{g}_A$.
We have
\begin{equation}\label{f:T(f)}
    (\oT^g_A f) (x) = A(x) (g * f) (x) \,.
\end{equation}
In particular, the corresponding action of $\oT_A^g$ is
\begin{equation}\label{f:action_T^g}
    \langle \oT^{g}_A a, b \rangle = \sum_z g(z) (\ov{b} \c a) (z) \,.
\end{equation}
for any functions $a,b : A \to \C$.
In the case $\ov{g(-x)} = g(x)$ by $\Spec (\oT^{g}_A)$ we denote the  spectrum of the hermitian operator $\oT^{g}_A$ (which is automatically real for hermitian matrices)
\begin{equation}\label{f:Spec_ordering}
    \Spec (\oT^{g}_A) = \{ \mu_1 \ge \mu_2 \ge \dots \ge \mu_{|A|} \} \,.
\end{equation}
Write  $\{ f \}_{\a}$, $\a\in [|A|]$ for the corresponding orthonormal eigenfunctions.
We call
%the number
$\mu_1$ as
the main eigenvalue and $f_1$ as the main function.
By the spectral theorem for hermitian matrices, see e.g. \cite{Horn-Johnson}, we have
\begin{equation}\label{f:spectral_decomposition}
        \oT^{g_1}_A (x,y) = \sum_{\a=1}^{|A|} \mu_\a f_\a (x) \ov{f_\a (y)} \,.
\end{equation}

Counting the trace of the operator $\oT^g_A$ and the trace of $\oT^g_A (\oT^g_A)^*$ one can easily obtains the formulae
\begin{equation}\label{f:trace_1}
    \tr(\oT^g_A) = \sum_{\a=1}^{|A|} \mu_\a = g(0) |A| \,.
\end{equation}
\begin{equation}\label{f:trace_2}
    \tr(\oT^g_A (\oT^g_A)^*) = \sum_{\a=1}^{|A|} |\mu_\a|^2 = \sum_{x,y} |\oT^g_A (x,y)|^2 = \sum_{z} |g(z)|^2 (A\c A) (z) \,.
\end{equation}
Using
%the formula above, which is
a particular case of the variational principle \cite{Horn-Johnson} one can estimate the the main eigenvalue of any hermitian (or even normal) matrix $\oT$
%there is a formula
\begin{equation}\label{f:variational}
    \mu_1 (\oT) = \max_{f ~:~ \| f \|_2 = 1}\, \langle \oT f, f \rangle \,.
\end{equation}
We will use formula (\ref{f:variational}) to bound the first eigenvalue of operators $\oT^g_A$.

Finally, it is easy to see that $\oT^B_A$ is a submatrix of the adjacency matrix of Caylay graph of $A$ with minus and $\t{\oT}^B_A$ is an ordinary submatrix.
Thus the operators $\oT^g_A$, $\t{\oT}^g_A$ can be considered as submatrices of the adjacency matrix of Caylay graph of $A$ with some weights.

\bigskip

Let us consider several examples of operators.

Our {\bf main example} the hermitian operator $\oT^{A\c A}_A$ which we denote by $\oT$ for brevity.
In the case, by variational principle (\ref{f:variational}) one has
\begin{equation}\label{f:T_1}
    \mu_1 (T) \ge \langle \oT A(x)/|A|^{1/2},  A(x)/|A|^{1/2} \rangle = \E(A) |A|^{-1} \,.
\end{equation}
Thus there is an obvious connection between the main eigenvalue of the operator and the additive energy.
Moreover, using formula (\ref{f:trace_2}), we see that
\begin{equation}\label{f:T_2}
    \sum_{\a=1}^{|A|} \mu^2_\a (\oT) = \sum_x (A\c A)^3 (x) = \E_3 (A) \,.
\end{equation}
Hence the energies $\E(A)$ and $\E_3 (A)$ being the second and the third moments of the convolution of the characteristic function are connected
%with
to
each other much more closely than just by the H\"{o}lder inequality.
Namely, they are connected through the operator $\oT$ and can be expressed via the eigenvalues of $\oT$.

Another simple observation is the following.
Consider a rectangular matrix $M(\v{x},y) = A(y) A^k(\v{x} + \D(y))$, $\v{x} \in A^k$, $y\in A$.
It is easy  to see that the {\it rectangular norm} of the matrix, that is
$$
    \| M \|^4_{\Box} := \sum_{\v{x}, \v{x}'} \sum_{y,y'} M(\v{x},y)\ov{M(\v{x}',y)}\, \ov{M(\v{x},y')} M(\v{x}',y')
$$
equals $\E_{2k+1} (A)$.
On the other hand, it is well--known that the rectangular norm is the fourth moment of the singular values of $M$ or, equivalently,  the sum of squares of the eigenvalues of the square
%operator
matrix
$(M M^*)(y,y') = (A\c A)^{k} (y-y') A(y) A(y')$.
Taking $k=1$, we obtain our operator $\oT$.
Thus $\oT$ is just the symmetrization of the matrix $M (\v{x},y)$, which is naturally and directly connected with higher energies.
Another (dual) view on such operators can be found in \cite{s}.

%Another results of such type can be found in \cite{Sh_energy}.

\bigskip

Let us formulate the main technical proposition of the section.

\begin{proposition}
    Let $A\subseteq \Gr$ be a set, $g_1,g_2$ be functions such that $\ov{g_1(-x)}= g_1(x)$,
    and $\{ f_\a \}$ be the orthonormal family of the eigenfunctions of the operator $\oT^{g_1}_A$.
    Then
\begin{equation}\label{f:triangles_g}
    \sum_{x,y,z\in A} g_1 (x-y) \ov{g_1 (x-z)}\, \ov{g_2 (y-z)}
        =
                \sum_{s,t} g_1 (s) \ov{g_1 (t)}\, \ov{g_2 (t-s)} \Cf_3 (A) (-s,-t)
                    =
\end{equation}
\begin{equation}\label{f:triangles_g'}
                    =
            \sum_{\a=1}^{|A|} \mu^2_\a (\oT^{g_1}_A) \cdot \ov{\langle \oT^{g_2}_A f_\a, f_\a\rangle} \,.
\end{equation}
In particular, if $g_1=g_2=g$ then
$$
    \sum_{x,y,z\in A} g (x-y) \ov{g (x-z)}\, \ov{g (y-z)}
        = \sum_{\a=1}^{|A|} |\mu_\a (\oT^{g}_A)|^2 \mu_\a (\oT^{g}_A)\,.
$$
\label{p:triangles_g}
\end{proposition}
\begin{proof}
    Let $\sigma$ be the first sum from (\ref{f:triangles_g}).
    Then the first identity in the formula can be obtained by the changing of the variables  $s=x-y$, $t=x-z$ and noting that the number of triples $(x,y,z) \in A^3$ with $s=x-y$, $t=x-z$ is exactly $\Cf_3 (A) (-s,-t)$.

    Let us prove identity (\ref{f:triangles_g'}).
    It is easy to see that
\begin{equation}\label{tmp:14.11.2015_2}
    \sigma = \sum_{x,y,z} \oT^{g_1}_A (x,y) \ov{\oT^{g_1}_A (x,z)}\, \ov{\oT^{g_2}_A (y,z)} \,.
\end{equation}
    By formula (\ref{f:spectral_decomposition}), we have
$$
        \oT^{g_1}_A (x,y) = \sum_{\a=1}^{|A|} \mu_\a f_\a (x) \ov{f_\a (y)} \,.
$$
    Substituting the last formula into (\ref{tmp:14.11.2015_2}) and changing the order of the summation, we obtain
$$
    \sigma = \sum_{y,z} \ov{\oT^{g_2}_A (y,z)} \sum_x \left( \sum_{\a=1}^{|A|} \mu_\a f_\a (x) \ov{f_\a (y)} \right)
        \left( \sum_{\a=1}^{|A|} \ov{\mu_\a} \ov{f_\a (x)} f_\a (z) \right)
    =
$$
$$
    =
        \sum_{y,z} \ov{\oT^{g_2}_A (y,z)} \sum_{\a,\beta} \mu_\a \ov{\mu_\beta} \ov{f_\a (y)} f_\a (z) \sum_x f_\a (x) \ov{f_\a (x)} \,.
$$
By the orthonormality of the eigenfunctions, we get
$$
    \sigma = \sum_\a |\mu_\a|^2 \sum_{y,z} \ov{\oT^{g_2}_A (y,z)}\, \ov{f_\a (y)} f_\a (z)
        =
            \sum_\a |\mu_\a|^2 \ov{\langle \oT^{g_2}_A f_\a, f_\a \rangle} \,.
$$
This completes the proof.
$\hfill\Box$
\end{proof}

\bigskip

\begin{exc}
    Using an appropriate generalization of the proposition above onto several operators, prove formula (\ref{f:uncerteinty}) of the previous section.
\end{exc}

\bigskip

Now we give two examples of operators with known spectrums and eigenfunctions.

{\bf Example I.}

Let $A\subseteq \Gr$ be a set, $D=A-A$, $S=A+A$, and put weight $g(x)$ equals $D(x)$ or $S(x)$.
In the following lemma we find, in particular, all eigenvalues as well as all eigenfunctions
of operators $\oT^{D}_{A}$, $\t{\oT}^{S}_{A}$.

\begin{lemma}
    Let $A\subseteq \Gr$ be a set, $D=A-A$, $S=A+A$.
    Then the main eigenvalues and eigenfunctions
    of the operators $\oT^{D}_{A}$, $\t{\oT}^{S}_{A}$
    %have
    equal
    $\mu_1 = |A|$, and $f_1(x) = A(x)/|A|^{1/2}$.
    All other eigenvalues equal zero and one can take for the correspondent eigenfunctions any orthonormal family  of functions on $A$ with zero mean.
\label{l:eigenvalues_D,S'}
\end{lemma}
\begin{proof}
We have
$$
    A(x) (D * A) (x) = A(x) |D \cap (A-x)| = |A| A(x)
$$
and thus by formula (\ref{f:T(f)}),  we see that $|A|$ is an eigenvalue of $\oT^{D}_{A}$ and $A(x) / |A|^{1/2}$ is the correspondent eigenfunction.
Applying (\ref{f:trace_2}), we have
$$
    |A|^2 + \sum_{\a=2}^{|A|} |\mu^2_\a (\oT^D_A)| = \sum_{\a=1}^{|A|} |\mu^2_\a (\oT^D_A)| = \sum_x D(x) (A\c A) (x) = |A|^2
$$
and hence all other eigenvalues of $\oT^D_A$ equal zero.
The arguments for $\t{\oT}^{S}_{A}$ are similar.
This concludes the proof.
$\hfill\Box$
\end{proof}

%Let us prove an asymmetric version of Lemma \ref{l:3/2_energy}.

\bigskip

Now we adapt the arguments from \cite{s_ineq}, see Proposition 28.

\begin{lemma}
    Let $A\subseteq \Gr$ be a finite set, $D=A-A$, $S=A+A$
    Suppose that $\psi$ be a
    %positive
    function on $\Gr$.
    Then
\begin{equation}\label{f:3/2_energy_D'}
    |A|^2 \left( \sum_x \psi (x) (A\c A) (x) \right)^2
        \le
            \E_3(A) \sum_x |\psi (x)|^2 (D\c D) (x)  \,,
\end{equation}
    and
\begin{equation}\label{f:3/2_energy_S'}
    |A|^2 \left( \sum_x \psi (x) (A\c A) (x) \right)^2
        \le
            \E_3(A) \sum_x |\psi (x)|^2 (S\c S) (x)  \,.
\end{equation}
\label{l:3/2_energy'}
\end{lemma}
\begin{proof}
    Let us prove (\ref{f:3/2_energy_S'}),
    the proof of (\ref{f:3/2_energy_D'}) is similar.
    Applying Proposition \ref{p:triangles_g} with $g_2 (x) = \psi (x)$ and $g_1 (x) = D(x)$,
    we obtain in view of Lemma \ref{l:eigenvalues_D,S'} that
$$
    \sum_{s,t} D(s) D(t) \psi (t-s) \Cf_3 (A) (-s,-t) = |A|^{-1} \mu^2_1 (\oT^D_A) \langle \oT^\psi_A A, A \rangle
        =
             |A|  \sum_x \psi (x) (A\c A) (x) \,.
$$
    Using the Cauchy--Schwarz inequality and formula (\ref{f:symmetry_E}) of Corollary \ref{c:symmetry_E}, we get
$$
    |A|^2 \left( \sum_x \psi (x) (A\c A) (x) \right)^2
        \le
            \E_3(A) \sum_x |\psi (x)|^2 (D\c D) (x)
$$
as required.
$\hfill\Box$
\end{proof}

\begin{exc}
    Obtain formulae (\ref{f:3/2_energy_D'}), (\ref{f:3/2_energy_S'}), using elementary arguments.
\end{exc}

\begin{corollary}
For any $A\subseteq \Gr$ the following holds
\begin{equation}\label{f:Li}
     |A|^2 \E^2_{3/2} (A)
        \le
            \E_3 (A) \E (A,A\pm A) \,.
%                \le
%                    \E^{1/3}_3 (A) \E^{2/3}_3 (B) \E (B,A\pm B) \,.
\end{equation}
\end{corollary}

%A slightly weaker inequality than (\ref{f:Li})
This inequality was obtained in \cite{Li} and implies formula (\ref{f:D_const}) from Lecture 2 without any additional multiplicative constants.

\begin{corollary}
For any $A\subseteq \Gr$ the following holds
\begin{equation*}\label{}
    |A|^6 \le \E_3 (A) \cdot \sum_{x\in A-A} ((A\pm A) \c (A\pm A)) (x) \,.
\end{equation*}
\end{corollary}

%Applying the previous Lemma with $A=B$ and $\psi=D$, we obtain
This  is an inequality from \cite{SS2}.

\bigskip

{\bf Example II.}

Let $p$ be a prime number, $q=p^s$ for some integer $s \ge 1$.
Let
$\F_q$ be the field with  $q$ elements, and let $\Gamma\subseteq \F_q$ be a multiplicative subgroup.
We will write $\F^*_q$ for $\F_q\setminus \{ 0 \}$.
Denote by $t$ the cardinality of $\Gamma$, and put $n=(q-1)/t$. Let
also $\mathbf{g}$ be a primitive root, then $\Gamma = \{ \mathbf{g}^{nl} \}_{l=0,1,\dots,t-1}$.
Let $\{ \chi_\a (x)\}_{\a \in [t]}$ be the
orthogonal
family of multiplicative characters on $\Gamma$ and
$\{ f_\a (x) \}_{\a \in [t]}$ be the correspondent orthonormal family, that is
\begin{equation}\label{f:chi_Gamma}
    f_\a (x) = |\G|^{-1/2} \chi_\a (x) = |\G|^{-1/2} \cdot e\left( \frac{\a l}{t} \right) \,, \quad x=\mathbf{g}^{nl} \,, \quad 0\le l < t \,.
\end{equation}
In particular, $f_\a (x) = \chi_\a (x) = 0$ if $x\notin \G$.
Clearly, products of such functions form a basis on Cartesian products of $\G$.

If $\_phi:\G \to \C$ be a function then denote by $c_\a (\_phi)$ the correspondent coefficients of $\_phi$ relatively to the family $\{ f_\a (x) \}_{\a \in [t]}$. In other words,
$$
    c_\a (\_phi) := \langle \_phi, f_\a \rangle = \sum_{x \in \G} \_phi (x) \ov{f_\a (x)} \,, \quad \quad \a \in [|\G|] \,.
$$

%The method of the paper based on the lemma,
In the next lemma we calculate, in particular, the spectrums of all operators with $\G$--invariant weights $g$ (that is $g(x \gamma) = g(x)$ for all $\gamma \in \G$.)
The lemma was proved mainly in \cite{SS1}.
We give the proof for the sake of completeness.
Further results on the spectrum of operators connected with multiplicative subgroups can be found in \cite{Sh_energy}.

\begin{lemma}
    Let $\Gamma\subseteq \F^*_q$ be a multiplicative subgroup.
    Suppose that $H(x,y) : \G \times \G \to \C$ satisfies two conditions
\begin{equation}\label{f:subgroup_eigenvalues}
    H(y,x) = \ov{H(x,y)} \quad \mbox{ and } \quad H(\gamma x, \gamma y) = H(x,y)\,, \quad \forall \gamma \in \G \,.
\end{equation}
    Then the functions $\{ f_\a (x) \}_{\a \in [|\G|]}$ form the complete orthonormal family of the eigenfunctions of the operator $H(x,y)$.
\label{l:subgroup_eigenvalues}
\end{lemma}
\begin{proof}
    The first property of (\ref{f:subgroup_eigenvalues}) says that $H$ is a hermitian operator, so it has a complete orthonormal family of its eigenfunctions.
    Consider the equation
\begin{equation}\label{tmp:25.04.2015_2}
    \mu f(x) = \G(x) \sum_{y\in \G} H(x,y) f(y) \,,
\end{equation}
    where $\mu$ is some number and $f : \G \to \C$ is unknown function.
    It is sufficient to check that any $f=\chi_\a$, $\a \in [|\G|]$ satisfies the equation above.
    Indeed, making a substitution $x\to x\gamma$ into (\ref{tmp:25.04.2015_2}) and using the characters property, we obtain
$$
    \mu f(x) f(\gamma) = \G(x\gamma) \sum_{y} H(\gamma x,y) f(y) = \G(x) \sum_{y} H(\gamma x,\gamma y) f(\gamma y)
        =
            \G(x) f(\gamma) \sum_{y} H(x,y) f(y) \,,
$$
    where the second property of (\ref{f:subgroup_eigenvalues}) has been used.
    Thus, it remains to check (\ref{tmp:25.04.2015_2}) just for one $x\in \G$.
    Choosing the number $\mu$ in an appropriate way we attain  the former.
    This completes the proof.
$\hfill\Box$
\end{proof}

\begin{corollary}
    Let $\G$ be a multiplicative subgroup.
    Then $\mu_1 (\oT) = \E(\G) |\G|^{-1}$ and
    $$
        \E(\G) = \max_{f ~:~ \| f\|_2 = 1,\, \supp f \subseteq \G} \E(\G,f) \,.
    $$
\end{corollary}
\begin{proof}
    Indeed by Lemma \ref{l:subgroup_eigenvalues} the function $f(x) := \G(x) / |\G|^{1/2}$ is the main eigenfunction of any operator $\oT^g_\G$ with $\G$--invariant function $g(x)$.
    In particular, $\langle \oT f, f\rangle =  \E(\G) |\G|^{-1}$.
    The second formula follows from variational principle (\ref{f:variational}).
$\hfill\Box$
\end{proof}

\begin{exc}
    Let $\G$ be a multiplicative subgroup.
    Then for any $A\subseteq \G$ one has
    $$
        \E(\G) \cdot \frac{|A|^2}{|\G|^2} \le \E(A,\G) \,,
    $$
    and
    $$
        \E^2 (A,\G) |\G| \le \E_3 (\G) \E^\times (A) \,.
    $$
\end{exc}

\bigskip

We give a number--theoretical application of Lemma \ref{l:subgroup_eigenvalues} above.

{\it Heilbronn's exponential sum} is defined by
\begin{equation}\label{def:Heilbronn_sum}
    S(a) = \sum_{n=1}^p e^{2 \pi i \cdot \frac{an^p}{p^2} } \,.
\end{equation}

D.R. Heath--Brown obtained in \cite{H} the first nontrivial upper bound for the sum, $a\neq 0$.
After that the result was improved in papers \cite{H-K}, \cite{s_heilbronn}, \cite{sv_heilbronn}, \cite{s_heilbronn2}.
%Here we obtain a new bound

Let us prove the best upper bound for the sum $S(a)$, see \cite{s_heilbronn2}.

Consider the following multiplicative subgroup
\begin{equation}\label{def:H_Gamma}
    \G = \{ m^p ~:~ 1\le m \le p-1 \} = \{ m^p ~:~ m \in \Z/(p^2 \Z) \,, m \neq 0 \} \subseteq \Z/(p^2 \Z) %\,.
\end{equation}
and note that
%$\max_{a\neq 0} |S(a)|$
$S(a)$
is just a sum over the subgroup $\G$.
We need in a lemma, which is analog of Corollary \ref{core3}.

\begin{lemma}
    For Heilbronn's subgroup (\ref{def:H_Gamma}), one has
\begin{equation}\label{f:E_3_Heilbronn}
    \E_3 (\G) \ll p^3 \log p \,.
\end{equation}
\label{l:E_3_Heilbronn}
\end{lemma}
\begin{proof}
Arranging $(\G \c \G) (x_1) \ge (\G \c \G) (x_2) \ge \dots $, where $x_j$ belong to different cosets, we have
by Lemma 7 from \cite{H-K} (see also Lemma 5 from \cite{s_heilbronn}) that
$$
    (\G \c \G) (x_j) \ll |\G|^{2/3} j^{-1/3} \,.
$$
Thus
$$
    \E_3 (\G) = \sum_{x} (\G \c \G)^3 (x) \ll |\G|^3 + |\G| \sum_{j} (|\G|^{2/3} j^{-1/3})^3
        \ll
            p^3 \log p
            %\,.
$$
as required.
%This completes the proof.
$\hfill\Box$
\end{proof}

\begin{theorem}
    Let $p$ be a prime, and $a\neq 0 \pmod p$.
    Then
    \begin{equation}\label{f:M_2/3}
        |S(a)| \ll p^{\frac{5}{6}} \log^{\frac{1}{6}} p \,.
    \end{equation}
\label{t:main-}
\end{theorem}
\begin{proof}
   Put $t = |\G| = p-1$.
    There is $\xi \neq 0$ such that
    \begin{equation}\label{tmp:20.11.2013}
        \oM^2 := |S(a)|^2 = t^{-1} \sum_{x\in \xi \G} |\FF{\G} (x)|^2
            = \langle \oT^{\FF{\xi \G}}_\G \G(x)/ t^{-1/2}, \G(x)/ t^{-1/2} \rangle \,.
    \end{equation}
    To derive the last identity we have used formulae (\ref{F_Par'}), (\ref{f:action_T^g}).
    Consider the operator $\oT^{\FF{\xi \G}}_\G$.
    Applying the Fourier transform or just simple calculations, it is easy to see that the operator is nonnegatively defined.
%    Moreover, by formula (\ref{f:action_T^g}) and the Fourier transform, we see that
%$$
%    t^{-1} \langle \oT^{\FF{\xi \G}}_\G \G, \G \rangle = t^{-1} \sum_{x\in \xi \G} |\FF{\G} (x)|^2 = \oM^2 \,.
%$$
    In view of  Lemma \ref{l:subgroup_eigenvalues} and identity (\ref{tmp:20.11.2013}), we obtain $\mu_1 (\oT^{\FF{\xi \G}}_\G) = \oM^2$.
    Using Proposition \ref{p:triangles_g} with $g_1 = g_2 = \xi \G$, further, nonnegativity of the operator $\oT^{\FF{\xi \G}}_\G$, Corollary \ref{c:symmetry_E} and the Cauchy--Schwarz inequality, we get
    %we obtain by Corollary  \ref{cor:E_l}
    $$
        \oM^{12} \le \E_3 (\G) \sum_{a,b} |\FF{\xi \G} (a)|^2 |\FF{\xi \G} (b)|^2 |\FF{\xi \G} (a-b)|^2 \,.
    $$
    Finally, applying formula (\ref{f:ET_2}) and Lemma \ref{l:E_3_Heilbronn}, we obtain
    $$
        \oM^{12} \le \E_3^2 (\G) p^2 \ll t^6 p^2 \log^2 t \,,
    $$
    and inequality (\ref{f:M_2/3}) follows.
    %This completes the proof.
$\hfill\Box$
\end{proof}

\bigskip

In two examples above we know eigenfunctions of the considered operators.
In our main example of the operator $\oT$ such functions are usually unknown.
Nevertheless, one can obtain some results in this general situation.
%and we give

%We begin
We start
with a lemma from \cite{s_mixed}, which shows that the operator $\oT^{A\c A}_A$ somehow "feels"\,  another operators $\oT^g_A$, $\t{\oT}^g_A$ for
%not nasty
"regular"\, weights
$g$.

\begin{lemma}
    Let $A\subseteq \Gr$ be a set and $g$ be a nonnegative function on $\Gr$.
    Suppose that $f_1$ is the main eigenfunction of $\oT^g_A$ or $\t{\oT}^g_A$,
    and $\mu_1$ is the correspondent eigenvalue.
    Then
$$
    \langle \oT^{A\c A}_A f_1 f_1 \rangle \ge \frac{\mu^3_1}{\| g\|_2^2 \cdot \| g\|_{\infty}} \,.
$$
\label{l:action_g}
\end{lemma}
\begin{proof}
By assumption $g$ is a nonnegative function on $\Gr$.
It implies that $f_1$ is also a  nonnegative function.
We have
\begin{equation}\label{tmp:29.11.2015_3}
    \mu_1 f_1 (x) = A(x) (g * f_1) (x) \,.
\end{equation}
Thus
\begin{equation}\label{tmp:29.11.2015_4}
    \mu_1^2 \left( \sum_x f_1 (x) \right)^2
        \le
            \left( \sum_x g(x) (f_1 \c A) (x) \right)^2
                \le
                    \| g\|_2^2 \E(A,f_1)
                        =
                            \| g\|_2^2 \langle \oT^{A\c A}_A f_1 f_1 \rangle \,.
\end{equation}
On the other hand, returning to (\ref{tmp:29.11.2015_3}) and using $\| f \|_2 =1$, we get
$$
    \mu_1 = \sum_x (f_1 \c f_1) (x) \le \| g\|_{\infty} \left( \sum_x f_1 (x) \right)^2 \,.
$$
Substituting the last estimate into (\ref{tmp:29.11.2015_4}), we obtain the result.
%and the result follows.
$\hfill\Box$
\end{proof}

\bigskip

By formula (\ref{f:T_1}) we know that $\mu_1 (\oT) \ge \E(A)/|A|$.
The next lemma shows that a similar upper bound holds if one consider large subsets of $A$.

\begin{lemma}
    Let $A\subseteq \Gr$ be a set.
    There is $A'\subseteq A$, $|A'| \ge |A|/2$
    such that $\mu_1 (\oT^P_{A'}) \le \frac{2\E(A)}{\Delta|A|}$
    for any set $P \subseteq \{ x ~:~ |A_x| \le \Delta \}$ and any real number $\Delta >0$.
    In particular, $\mu_1 (\oT^{A\c A}_{A'}) \le \frac{2\E(A)}{|A|}$.
    %and $\mu_0 (\oT^{P(A\c A)}_{A'}) \le \frac{2\E(A)}{|A|}$.
\label{l:A'_0.5}
\end{lemma}
\begin{proof}
Let
$$
    A_1 = \{ x ~:~ ((A*A)\c A) (x) > 2\E(A)/|A| \} \,.
$$
It is easy to see that $|A_1| < |A|/2$.
Put $A' = A\setminus A_1$ and let $f$ be the main eigenfunction of the operator $\oT^P_{A'}$.
Let also $\mu_1 = \mu_1 (\oT^P_{A'})$.
We have
$$
    \mu_1 f(x) = A'(x) (P* f)(x) \,.
$$
Summing over $x\in A'$ and using the definition of the set $A'$, we obtain
$$
    \mu_1 \sum_x f(x) = \sum_x f(x) (P\c A') (x) \le \Delta^{-1} \sum_x f(x) ((A\c A) \c A ) (x)
        =
$$
$$
    =   \Delta^{-1} \sum_x f(x) ((A*A)\c A) (x)
        \le
            \Delta^{-1} \frac{2\E(A)}{|A|} \cdot \sum_x f(x)
$$
and we are done.
$\hfill\Box$
\end{proof}

\bigskip

    There is an important class of so--called connected sets.
    Formally, let $\beta,\gamma \in [0,1]$.
    A set $A\subseteq \Gr$ is called {\it $(\beta,\gamma)$--connected}
    if for any $B \subseteq A$, $|B| \ge \beta|A|$
    %one has
    the following holds
    $$
        \E (B) \ge \gamma \left( \frac{|B|}{|A|} \right)^{4} \E (A) \,.
    $$

Using Lemma \ref{l:A'_0.5}, one can obtain an unusual
%connection
relation
between energies $\E_s (A)$, $s\in [1,2]$ and $\E(A)$
for any connected set $A$, see \cite{s_mixed}.

\begin{exc}
    Let $A\subseteq \Gr$ be a set, and $\beta,\gamma \in [0,1]$.
    Suppose that $A$ is $(\beta,\gamma)$--connected with $\beta \le 1/2$.
    Further for any $s\in [1,2]$ the following holds
\begin{equation}\label{f:connected}
    \E_s (A) \ge 2^{-5} \gamma |A|^{1-s/2} \E^{s/2} (A) \,.
\end{equation}
\end{exc}

\bigskip

Using lemmas above we can formulate the second main result of the section (another theorems of such type can be found in \cite{BK_struct}, \cite{s_mixed}).
By formulas (\ref{f:T_1}), (\ref{f:T_2}) we know that the energies $\E(A)$, $\E_3(A)$ are connected through the operator $\oT$. It allows us give a full description of sets $A$ having "critical relation"\, between $\E(A)$, $\E_3(A)$ that is $\E_3 (A) \ll \E^2 (A) / |A|^2$.
Namely, inequality $\E_3 (A) \ll \E^2 (A) / |A|^2$ holds iff $A$ contains a large subset $A'$ such that $|nA'-mA'| \approx |A'-A'|$ for any positive integers $n,m$.
Informally, it says that the growth of the size of the sumset $kA'$ of the set $A'$ stops after the second step.
More precisely, the following holds (previous results in the direction can be found in \cite{SS1} and \cite{s_ineq}).

\begin{theorem}
    Let $A\subseteq \Gr$ be a set, $\E (A) = |A|^{3}/K^{}$,
    and $\E_3 (A) = M|A|^4 / K^2$.
    Then there is a set $A' \subseteq A$
    such that
    \begin{equation}\label{f:E_3_size}
        |A'| \gg  M^{-10} \log^{-15} M \cdot |A|  \,,
    \end{equation}
    and
    \begin{equation}\label{f:E_3_doubling}
        |nA'-mA'| \ll (M^{9} \log^{14} M)^{6(n+m)} K |A'|
    \end{equation}
    for every $n,m\in \N$.
\label{t:E_3_M}
\end{theorem}
\begin{proof}
Let
$\E = \E (A) = |A|^{2} / K^{}$,
$\E_3 = \E_3 (A)$, $L=2 \log (4M)$.
%Because of $\E_3$ is small we can apply the arguments from the proof of Theorem \ref{t:energy_gen}.
Write
$$
    D_j = \{ x\in A-A ~:~ 2^{j-2} |A| K^{-1} < |A_x| \le 2^{j-1} |A| K^{-1} \} \,.
$$
Trivially
$$
    |D_j| (2^{j-2} |A| K^{-1})^3 \le \E_3 \,,
$$
and whence
\begin{equation}\label{tmp:28.07.2012_1*}
    |D_j| \ll \E_3 / (|A|^3 K^{-3} 2^{3j}) \,.
\end{equation}
Thus
$$
    \E  \ll \sum_{j=1}^l \sum_{s} |A_s|^2 \,,
$$
where $l$ can be estimated as $\log M^{} \le L$.
By pigeonhole principle we find $j\in [l]$ such that
\begin{equation}\label{tmp:17.11.2012_D&}
    L^{-1} \E  \ll \sum_{s\in D_j} |A_s|^2 \,.
\end{equation}
Put $D=D_j$, $\Delta = 2^{j-1} |A| K^{-1}$, and $g(x) = (A\c A)^{} (x) D(x)$.
From (\ref{tmp:17.11.2012_D&}) it follows that
\begin{equation}\label{tmp:17.11.2012_D_and_tilde}
    |D| \gg \frac{|A|^{} K^{}}{L M^{2}}
\end{equation}
and
\begin{equation}\label{tmp:17.11.2012_D_and_tilde'}
%        \quad \mbox{ and } \quad
            \sum_{x\in D} (A\c A) (x) \gg \frac{|A|^2}{L M^{}} \,.
\end{equation}
Consider the operators $\oT_1 = \oT^g_{A}$,
$\oT_2 = \oT^{A}_{A,D}$ and $\oT_3 = \oT^{A\c A}_A$.
Using Lemma \ref{l:action_g}, we get
\begin{equation}\label{tmp:TILDE}
    \langle \oT_3 f_1, f_1 \rangle \ge \frac{\mu^3_0 (\oT_1)}{\| g\|^2_2 \| g \|_\infty}
        \gg
            \frac{|D|^2 \D^3}{|A|^3} := \sigma \,.
\end{equation}
Clearly, all elements of matrices $\oT_1, (\oT_2)^* \oT_2$
%are less than
does not exceed elements of
$\oT_3$ and
the operator $\oT_3$ is nonnegative  defined.
By formula (\ref{tmp:17.11.2012_D&}), we have
\begin{equation}\label{tmp:17.11.2012_D}
    \frac{\E}{4L|A|} \le \mu_1 (\oT_1) \,.
    %,~ \la^2_0 (\oT_2) \,.
\end{equation}
Similarly,
\begin{equation}\label{tmp:17.11.2012_D'}
     %\frac{\E}{|A|} \le \mu_0 (\oT_3) \,.
     \frac{\E}{4L|A|}
        \le
            %\max\{ \mu_0 (\oT_1), \| \oT_2 f_0 \|^2_2 \}
                \mu_1 (\oT_1)
                \le \langle \oT_3 f_1 , f_1 \rangle \,,
\end{equation}
where $f_1 \ge 0$ is the main eigenfunction of the operator $\oT_1$.
%Denote by $m$ the maximum from  (\ref{tmp:17.11.2012_D'}).
Applying Proposition \ref{p:triangles_g} with $A=A$, $g_1=g$, $g_2 = A\c A$,
%and nonnegativity of $\oT_3$,
we obtain
$$
    \mu^4_1 (\oT_1) \sigma^2
        \ll
            \E_3
                \sum_{\a \in D,\beta \in D ~:~ (A\c A) (\a-\beta) \ge d}
                    (A\c A)^{2} (\a) (A\c A)^{2} (\beta) (A\c A)^2 (\a-\beta)
$$
\begin{equation}\label{tmp:20.11.2012_ev}
    \ll
            \E_3 \D^{4} \sum_{x ~:~ (A\c A) (x) \ge d} (D\c D) (x) (A\c A)^2 (x) \,,
\end{equation}
(where $d$ can be taken as $d=\frac{\mu^2_0 (\oT_1)}{32|A|\E^{1/2}_3}$).
Applying the Cauchy--Schwartz inequality, we have
$$
    \sum_x (A\c A)^2 (x) (D \c D) (x) \le \E^{2/3}_3 \left( \sum_{x} (D \c D)^3 (x) \right)^{1/3}
        \le
            \E^{2/3}_3 |D|^{1/3} \E^{1/3} (D) \,.
$$
Put $\E (D) = \mu |D|^3$.
Recalling (\ref{tmp:20.11.2012_ev}), we get
\begin{equation}\label{f:mu_sigma}
    \mu^4_1 (\oT_1) \sigma^2
        \ll
            \left( \frac{M |A|^4}{K^2} \right)^{5/3} \D^{4} |D|^{4/3} \mu^{1/3} \,.
\end{equation}
We have $\D \ll M^{} |A| / K$.
In the situation the following holds $\sigma \ge \mu_1 (\oT_1)$.
Thus,
an
accurate calculations give
$$
    \E (D) = \mu |D|^3 \gg \frac{|D|^3}{M^9 L^{14}} \,.
$$
By Balog--Szemer\'{e}di--Gowers Theorem, see e.g. \cite{tv},
%and estimate (\ref{tmp:05.08.2012_1})
there is
$D' \subseteq D$ such that
$|D'| \gg \mu |D|$
and
$
    |D'+D'| \ll \mu^{-6} |D'|
$.
Pl\"{u}nnecke--Ruzsa inequality (see \cite{p} or again \cite{tv}) yields
\begin{equation}\label{tmp:31.07.2012_1}
    |nD'-mD'| \ll \mu^{-6(n+m)} |D'| \,,
\end{equation}
for every $n,m \in \N$.
Using the definition of the set $D=D_j$ and inequality
%$\D \ll M |A| / K$
(\ref{tmp:17.11.2012_D_and_tilde'}),
we find $x\in \Gr$ such that
\begin{equation}\label{tmp:31.07.2012_2}
    |(A-x) \cap D'| \gg \mu |A| L^{-1} M^{-1}
                %\gg 2^{j} K^{-1} \mu |D|
        \gg
            M^{-10} L^{-15} \cdot |A| \,.
\end{equation}
Put $A' = A\cap (D'+x)$.
Using (\ref{tmp:31.07.2012_1}), (\ref{tmp:31.07.2012_2}) and the definition of $\D$,
we obtain for all $n,m \in \N$
\begin{equation}\label{tmp:31.07.2012_2'''}
    |nA'-mA'| \le |nD'-mD'| \ll \mu^{-7(n+m)} |A| |A'| \D^{-1}
        \ll
            \mu^{-6(n+m)} K |A'|
\end{equation}
and the theorem is proved.
%This completes the proof.
$\hfill\Box$
\end{proof}

\begin{remark}
For every convex set
%and a small multiplicative subgroup $A$ of $\F_p$
Theorem \ref{t:E_3_M} above easily gives a "nontrivial"\, estimate $\E(A) \ll |A|^{5/2-\eps_0}$,
where $\eps_0>0$ is an absolute constant.
Indeed, suppose that $\E(A) \gg |A|^{5/2-\eps}$ and $\eps>0$ is sufficiently small.
Then recalling the bound $\E_3 (A) \ll |A|^3 \log |A|$, we see that
in terms of  Theorem \ref{t:E_3_M} one has $M = \log |A|$.
So, $M$ is small and we can effectively apply the theorem.
Thus there is a set $A' \subseteq A$ from Theorem \ref{t:E_3_M} such that
\begin{equation}\label{tmp:30.11.2015_1}
    |A|^{7/4} \ll_M |A'+A'-A'| \ll_M |A|^4 \E^{-1} (A)
\end{equation}
and the result follows.
In the derivation of the first inequality of (\ref{tmp:30.11.2015_1}) we have used formula (\ref{f:A'_B}) of Exercise \ref{ex:A'_B}.

Applying more refine method from \cite{SS2} one can get even simpler proof.
Indeed, for so large $A'\subseteq A$ we have (see Exercise \ref{ex:A'_B}) that
$|A|^{3/2+\eps_1} \ll_M |A'-A'| \ll_M |A|^4 \E^{-1} (A)$, where
$\eps_1>0$ is an absolute constant.
Again we obtain a lower bound for $\eps_0$.
Interestingly, that lower bounds for the doubling constants
give us upper bounds for the additive energy in the case.
Of course our real
%proof
arguments even more direct
%is more direct
and they give a concrete  bound
\begin{equation}\label{tmp:30.11.2015_2}
    \E(A) \ll |A|^{32/13+\eps} \,, \quad \quad \eps>0
\end{equation}
for any convex set $A$.

The same proof takes place for multiplicative subgroups
$\G \subseteq \Z/p\Z$, where $p$ is a prime number
if one use Stepanov's method (see e.g. \cite{K_Tula} or \cite{SV})
or combine Stepanov's method with
%new
recent
lower bounds for the doubling constant of subgroups from \cite{SS1,SV,s_ineq}.
The direct arguments give (\ref{tmp:30.11.2015_2}) for any subgroup $\G$ of size less than $\sqrt{p}$, say.
\label{r:worker-peasant}
\end{remark}

Theorem \ref{t:E_3_M} describes all sets $A$, having "critical"\, relation between energies $\E_2 (A)$, $\E_3(A)$.
Another pairs of energies were considered in papers \cite{BK_struct} and \cite{s_mixed}.

%\bigskip

%\begin{conjecture}
%    May be for some specific $A$, e.g. $A$ has the form $B-B$ or a subset of large values of $(B\c B) (x)$?
%\end{conjecture} 

\bigskip

\noindent{I.D.~Shkredov\\
Steklov Mathematical Institute,\\
ul. Gubkina, 8, Moscow, Russia, 119991}
%MSU, IPPI RAN\\}
%%\\
%and
%\\
%Delone Laboratory of Discrete and Computational Geometry,\\
%Yaroslavl State University,\\
%Sovetskaya str. 14, Yaroslavl, Russia, 150000
\\
and
\\
IITP RAS,  \\
Bolshoy Karetny per. 19, Moscow, Russia, 127994\\
{\tt ilya.shkredov@gmail.com}

\end{document}